\documentclass[a4paper,10pt,reqno]{amsart}
\usepackage{amsmath,amsfonts,amsthm,amssymb,color}
\usepackage[T1]{fontenc}
\usepackage{pdfsync}
\usepackage{csquotes}
\usepackage{graphicx}
\usepackage{pstricks}
\usepackage{lmodern}
\usepackage{mathabx}

\usepackage{tikz}

\newcommand{\com}[1]{\vskip.3cm
\hspace{0.6cm}\parbox{0.91\linewidth}{ #1}
\vskip.3cm}

\newcommand{\be}{\beta}

\newcommand{\1}{\mathbf{1}}

\newcommand{\rti}{\tilde{r}}
\newcommand{\tti}{\tilde{t}}
\newcommand{\sti}{\tilde{s}}
\newcommand{\uti}{\tilde{u}}
\newcommand{\vti}{\tilde{v}}
\newcommand{\xiti}{\tilde{\xi}}
\newcommand{\etati}{\tilde{\eta}}
\newcommand{\lati}{\tilde{\la}}

\newcommand{\yti}{\tilde{y}}
\newcommand{\zti}{\tilde{z}}

\newcommand{\R}{\mathbb R}

\newcommand{\cac}{\mathcal C}
\newcommand{\cd}{\mathcal D}
\newcommand{\ce}{\mathcal E}
\newcommand{\cf}{\mathcal F}
\newcommand{\cg}{\mathcal G}
\newcommand{\ch}{\mathcal H}
\newcommand{\ci}{\mathcal I}
\newcommand{\cj}{\mathcal J}
\newcommand{\cl}{\mathcal L}

\newcommand{\al}{\alpha}
\newcommand{\si}{\sigma}

\newcommand{\ga}{\gamma}

\newcommand{\ka}{\kappa}
\newcommand{\la}{\lambda}

\newcommand{\vp}{\varphi}

\newtheorem{theorem}{Theorem}[section]

\newtheorem{definition}[theorem]{Definition}

\newtheorem{lemma}[theorem]{Lemma}
\newtheorem{notation}[theorem]{Notation}

\newtheorem{proposition}[theorem]{Proposition}

\theoremstyle{remark}
\newtheorem{remark}[theorem]{Remark}

\pgfdeclareshape{crosscircle}
{
  \inheritsavedanchors[from=circle] 
  \inheritanchorborder[from=circle]
  \inheritanchor[from=circle]{north}
  \inheritanchor[from=circle]{north west}
  \inheritanchor[from=circle]{north east}
  \inheritanchor[from=circle]{center}
  \inheritanchor[from=circle]{west}
  \inheritanchor[from=circle]{east}
  \inheritanchor[from=circle]{mid}
  \inheritanchor[from=circle]{mid west}
  \inheritanchor[from=circle]{mid east}
  \inheritanchor[from=circle]{base}
  \inheritanchor[from=circle]{base west}
  \inheritanchor[from=circle]{base east}
  \inheritanchor[from=circle]{south}
  \inheritanchor[from=circle]{south west}
  \inheritanchor[from=circle]{south east}
  \inheritbackgroundpath[from=circle]
  \foregroundpath{
    \centerpoint%
    \pgf@xc=\pgf@x%
    \pgf@yc=\pgf@y%
    \pgfutil@tempdima=\radius%
    \pgfmathsetlength{\pgf@xb}{\pgfkeysvalueof{/pgf/outer xsep}}%
    \pgfmathsetlength{\pgf@yb}{\pgfkeysvalueof{/pgf/outer ysep}}%
    \ifdim\pgf@xb<\pgf@yb%
      \advance\pgfutil@tempdima by-\pgf@yb%
    \else%
      \advance\pgfutil@tempdima by-\pgf@xb%
    \fi%
    \pgfpathmoveto{\pgfpointadd{\pgfqpoint{\pgf@xc}{\pgf@yc}}{\pgfqpoint{-0.707107\pgfutil@tempdima}{-0.707107\pgfutil@tempdima}}}
    \pgfpathlineto{\pgfpointadd{\pgfqpoint{\pgf@xc}{\pgf@yc}}{\pgfqpoint{0.707107\pgfutil@tempdima}{0.707107\pgfutil@tempdima}}}
    \pgfpathmoveto{\pgfpointadd{\pgfqpoint{\pgf@xc}{\pgf@yc}}{\pgfqpoint{-0.707107\pgfutil@tempdima}{0.707107\pgfutil@tempdima}}}
    \pgfpathlineto{\pgfpointadd{\pgfqpoint{\pgf@xc}{\pgf@yc}}{\pgfqpoint{0.707107\pgfutil@tempdima}{-0.707107\pgfutil@tempdima}}}
  }
}
\makeatother

\colorlet{symbols}{blue!90!black}
\colorlet{testcolor}{green!60!black}

\usetikzlibrary{shapes.misc}
\usetikzlibrary{shapes.symbols}
\usetikzlibrary{snakes}
\usetikzlibrary{decorations}
\usetikzlibrary{decorations.markings}

\def\drawx{\draw[-,solid] (-3pt,-3pt) -- (3pt,3pt);\draw[-,solid] (-3pt,3pt) -- (3pt,-3pt);}
\tikzset{
	root/.style={circle,fill=testcolor,inner sep=0pt, minimum size=2mm},
	dot/.style={circle,fill=black,inner sep=0pt, minimum size=1mm},
	var/.style={circle,fill=black!10,draw=black,inner sep=0pt, minimum size=2mm},
	dotred/.style={circle,fill=black!50,inner sep=0pt, minimum size=2mm},
	generic/.style={semithick,shorten >=1pt,shorten <=1pt},
	dist/.style={ultra thick,draw=testcolor,shorten >=1pt,shorten <=1pt},
	testfcn/.style={ultra thick,testcolor,shorten >=1pt,shorten <=1pt,<-},
	testfcnx/.style={ultra thick,testcolor,shorten >=1pt,shorten <=1pt,<-,
		postaction={decorate,decoration={markings,mark=at position 0.6 with {\drawx}}}},
	kprime/.style={semithick,shorten >=1pt,shorten <=1pt,densely dashed,->},
	kprimex/.style={semithick,shorten >=1pt,shorten <=1pt,densely dashed,->,
		postaction={decorate,decoration={markings,mark=at position 0.4 with {\drawx}}}},
	kernel/.style={semithick,shorten >=1pt,shorten <=1pt,->},
	multx/.style={shorten >=1pt,shorten <=1pt,
		postaction={decorate,decoration={markings,mark=at position 0.5 with {\drawx}}}},
	kernelx/.style={semithick,shorten >=1pt,shorten <=1pt,->,
		postaction={decorate,decoration={markings,mark=at position 0.4 with {\drawx}}}},
	kernel1/.style={->,semithick,shorten >=1pt,shorten <=1pt,postaction={decorate,decoration={markings,mark=at position 0.45 with {\draw[-] (0,-0.1) -- (0,0.1);}}}},
	kernel2/.style={->,semithick,shorten >=1pt,shorten <=1pt,postaction={decorate,decoration={markings,mark=at position 0.45 with {\draw[-] (0.05,-0.1) -- (0.05,0.1);\draw[-] (-0.05,-0.1) -- (-0.05,0.1);}}}},
	kernelBig/.style={semithick,shorten >=1pt,shorten <=1pt,decorate, decoration={zigzag,amplitude=1.5pt,segment length = 3pt,pre length=2pt,post length=2pt}},
	rho/.style={dotted,semithick,shorten >=1pt,shorten <=1pt},
	renorm/.style={shape=circle,fill=white,inner sep=1pt},
	labl/.style={shape=rectangle,fill=white,inner sep=1pt},
	xi/.style={circle,fill=symbols!10,draw=symbols,inner sep=0pt,minimum size=1.2mm},
	xix/.style={crosscircle,fill=symbols!10,draw=symbols,inner sep=0pt,minimum size=1.2mm},
	xib/.style={circle,fill=symbols!10,draw=symbols,inner sep=0pt,minimum size=1.6mm},
	xibx/.style={crosscircle,fill=symbols!10,draw=symbols,inner sep=0pt,minimum size=1.6mm},
	not/.style={circle,fill=symbols,draw=symbols,inner sep=0pt,minimum size=0.5mm},
	>=stealth,
	}
\makeatletter
\def\DeclareSymbol#1#2#3{\expandafter\gdef\csname MH@symb@#1\endcsname{\tikz[baseline=#2,scale=0.15,draw=symbols]{#3}}\expandafter\gdef\csname MH@symb@#1s\endcsname{\scalebox{0.7}{\tikz[baseline=#2,scale=0.15,draw=symbols]{#3}}}}
\def\<#1>{\csname MH@symb@#1\endcsname}
\makeatother

\DeclareSymbol{circle}{0.5}{\draw (0,0.7) node[xi] {};}
\DeclareSymbol{line}{0.5}{\draw (0,0.2) node[not] {} -- (0,1.4) node[not] {};}

\DeclareSymbol{Psi}{0.5}{\draw (0,0) node[not] {} -- (0,1.5) node[xi] {};}
\DeclareSymbol{IPsi}{0}{\draw (0,0.8) -- (0,2) node[xi] {};\draw (0,-0.25) node[not] {} -- (0,0.8) node[not] {} ;}
\DeclareSymbol{Psi2}{0.5}{\draw (-0.8,1) node[xi] {} -- (0,0) node[not] {} -- (0.8,1) node[xi] {};}
\DeclareSymbol{IPsi2}{0}{\draw (0,0.8) -- (0.8,1.9) node[xi] {};\draw (0,-0.4) node[not] {} -- (0,0.8) node[not] {} -- (-0.8,1.9) node[xi] {};}
\DeclareSymbol{PsiIPsi2}{0}{\draw (0,1) -- (0.8,2.2) node[xi] {};\draw (0,-0.25) node[not] {} -- (0,1) node[not] {} -- (-0.8,2.2) node[xi] {};\draw (0,-0.25) node[not] {} -- (0.8,1) node[xi] {};}

\date{\today}

\title{On ill-posedness of nonlinear stochastic wave equations driven by rough noise}

\begin{document}

\maketitle

\begin{center}
{\large
Aur\'elien Deya\footnote{Institut \'Elie Cartan, Universit\' e de Lorraine, BP 70239, 54506 Vandoeuvre-l\`es-Nancy, France. \\Email: {\tt aurelien.deya@univ-lorraine.fr}}}
\end{center}

\bigskip

\bigskip

\noindent 
{\bf Abstract.} We highlight a fundamental ill-posedness issue for nonlinear stochastic wave equations driven by a fractional noise. Namely, if the noise becomes too rough (i.e., the sum of its Hurst indexes becomes too small), then there is essentially no hope to provide an interpretation of the model, whether directly or through a Wick-type renormalization procedure. This phenomenon can be compared with the situation of a general SDE driven by a two-dimensional fractional noise of index $H\leq \frac14$. 

\smallskip

Our results clarify and extend previous similar properties exhibited in \cite{De2} or in \cite{oh-okamoto}.

\bigskip


\bigskip


\section{Introduction}\label{sec:intro}


The present study is devoted to the investigations of ill-posedness issues related to the stochastic nonlinear wave model:
\begin{equation}\label{equa-intro}
\begin{cases}
\partial^2_t u=\Delta u+f(u)+\si(u) \dot{B}, \quad t\in [0,T], \ x\in \R^d\, ,\ d\geq 1\\
u_0=(\partial_t u)_0=0\, ,
\end{cases}
\end{equation}
where $f,\si:\R\to \R$ are regular maps satisfying $f''(0)\neq 0$, $\si(0)\neq 0$, and $\dot{B}$ stands for a space-time fractional noise (which includes the case of space-time white noise). Roughly speaking, we intend to exhibit some breakup phenomenon with respect to the regularity of $\dot{B}$, by showing that below a specific threshold, no interpretation of the model - whether classical, It{\^o} or Wick-type - can be expected, owing to a fundamental singularity issue. 

\

Such a sudden failure in the analysis - with a noise becoming \enquote{too rough} - is a well-identified phenomenon for standard stochastic differential systems, i.e. when dealing with the one-parameter model
\begin{equation}\label{back-to-sde}
\begin{cases}
dY_t=\si(Y_t)\, \dot{B}_t\, ,  \quad  t\in [0,T],\\
Y_0=0\, ,
\end{cases}
\end{equation}
where $\si:\R \to \cl(\R^m,\R)$ is a regular map and $\dot{B}=dB$ stands for a $m$-dimensional fractional noise of Hurst index $H\in (0,1)$ (see Definition \ref{defi:frac-noise} for a general presentation of fractional noises). Indeed, it has long been established that for $m\geq 2$, a robust interpretation and analysis of \eqref{back-to-sde} is possible as long as $H> \frac14$, while no fully satisfying general interpretation is available when $H\leq \frac14$. The theory of rough paths (\cite{friz-hairer,friz-victoir,lyons-qian}) offers a very clear insight on such a behaviour. In brief, the theory somehow allows us to reduce the analysis of \eqref{back-to-sde} to the study of a finite-order development of the equation, that is one can morally focalize on the approximation 
\begin{align}
Y_t &= \sum_{i=1}^m \int_0^t \si_i(Y_s) \,  \dot{B}^{(i)}_s\approx \sum_{i=1,\ldots,m} \si_i(0)\int_0^t \dot{B}^{(i)}_s+\sum_{i,j=1,\ldots,m} \si'_i(0)\si_j(0)\int_0^t \int_0^s \dot{B}^{(j)}_r \dot{B}^{(i)}_s+...\label{expansion-sde}
\end{align}
where the (finite) order of the expansion depends on the regularity of $B$. In particular, rough paths theory emphasizes the central role of the iterative integrals $\big(\int_0^t \dot{B}^{(i)},\int_0^t\int_0^s \dot{B}^{(j)}_r\dot{B}^{(i)}_s$,...) in the dynamics: the possibility to construct these objects becomes a necessary and sufficient condition in order to interpret and solve the equation for all regular vector fields $\si$. Keeping these principles in mind, the breakup phenomenon for the standard equation \eqref{back-to-sde} can be precisely expressed as follows.

\begin{proposition}\label{prop:intro-fbm}
Let $\dot{B}=(\dot{B}^{(1)},\dot{B}^{(2)})$ be a two-dimensional fractional noise on $\R$, with index $H\in (0,1)$. Consider a smooth and compactly-supported function $\rho:\R\to \R_+$ such that $\int_{\R} \rho=1$, and for every $n\geq 0$, set $\rho_n(x):=2^n \rho(2^n x)$. Finally, denote by $\dot{B}^{(i),n}$ the mollification of $\dot{B}^{(i)}$ through $\rho_n$, i.e. 
\begin{equation}\label{mollifi-noi}
\dot{B}^{(i),n}:=\rho_n\ast \dot{B}^{(i)}.
\end{equation} 
In this setting, the following picture holds true:

\smallskip

\noindent
$(i)$ If $H\in (\frac14,1)$, then for all $1\leq i,j\leq 2$, the sequence of processes $\big(t\mapsto\int_0^t ds \int_0^s dr\, \dot{B}^{(i),n}_r \, \dot{B}^{(j),n}_s\big)_{n\geq 0}$ converges almost surely in the space $\cac([0,T];\R)$ of continuous functions.

\smallskip

\noindent
$(ii)$ If $0<H\leq \frac14$, then for any non-zero smooth function $\vp:\R_+\to \R_+$ with support in $[0,1]$, one has
\begin{equation}\label{explosion-levy}
\mathbb{E}\bigg[\Big|\big\langle \int_0^. ds\int_0^sdr \, \dot{B}^{(1),n}_r \dot{B}^{(2),n}_s,\vp\big\rangle \Big|^2\bigg] \stackrel{n\to \infty}{\longrightarrow} \infty\, .
\end{equation}
\end{proposition}

The above general behaviour of the noise family $(\dot{B}^H)_{H\in (0,1)}$ has been first highlighted by Coutin and Qian in their work \cite[Theorem 2]{coutin-qian}. This being said, we are not aware of any previous \enquote{explosion} statement as general as the one in item $(ii)$, and for this reason we have included the proof of \eqref{explosion-levy} in Appendix \ref{append:levy-area} (the latter can also be seen as an introduction to the more sophisticated computations in the wave setting). 

\smallskip

The statement of Proposition \ref{prop:intro-fbm} thus reads as follows. As long as $H>\frac14$, one is able to provide a  natural interpretation of the so-called L{\'e}vy-area term $\int_0^. \int_0^s\dot{B}_r \otimes \dot{B}_s$, which in turn can be injected into the rough paths machinery, opening the way for a full treatment of equation \eqref{back-to-sde}. On the other hand, as soon as $H\leq \frac14$, the situation regarding this L{\'e}vy-area term (and subsequently the equation itself) becomes totally out of control: according to item $(ii)$, not only does $t\mapsto\int_0^t \int_0^s\dot{B}^{(1),n}_r  \dot{B}^{(2),n}_s$ fail to converge in the space of continuous functions, but it also explodes as a general distribution of time, ruling out any potential interpretation of \eqref{back-to-sde} (at least for non-trivial $\sigma$).

\

With Proposition \ref{prop:intro-fbm} in mind, our objective in the sequel is now easy to state. Namely, we would like to exhibit a similar drastic change of regime for the more sophisticated SPDE dynamics \eqref{equa-intro}, as the regularity of the noise $\dot{B}$ decreases. In order to - heuristically - find out the exact counterpart of Proposition \ref{prop:intro-fbm} in the wave setting, let us follow a similar approach as with the standard model \eqref{back-to-sde} and expand equation \eqref{equa-intro}, starting from its mild formulation
\begin{equation}\label{mild-intro}
u(t,x)=\int_0^t \int_{\R^d}\cg_{t-s}(x-y)f(u(s,y)) \, dy ds+\int_0^t \int_{\R^d} \cg_{t-s}(x-y)\si(u(s,y)) \dot{B}(ds,dy) \, .
\end{equation}
Throughout the paper, the notation $\cg$ will refer to the fundamental solution of the wave equation on $\R^d$, characterized by its spatial Fourier transform
\begin{equation}\label{fourier-wave-kernel}
\cf_x(\cg_t)(\eta):=\int_{\R^d} dx \, e^{-\imath \langle \xi,x\rangle}\cg_t(x)=\frac{\sin (t|\eta|)}{|\eta|}\, , \quad t\geq 0, \ \eta\in \R^d.
\end{equation}
Setting additionally $\cg_t:=0$ and $\dot{B}_t:=0$ for $t\leq 0$, equation \eqref{mild-intro} can be more concisely rephrased as
$$u=\cg\ast \big(\si(u)\dot{B}\big)+\cg \ast \big(\1_{\R_+} f(u)\big),$$
where the notation $\ast$ stands for the convolution in $\R^{d+1}$. Now let us stick here to a local expansion of $u$ around the initial condition, i.e. zero, by taking the nonlinear assumptions $f''(0)\neq 0,\si(0)\neq 0$ into account. In other words, let us approximate $f(u)$ and $\si(u)$ by 
$$f(u)\approx f(0)+f'(0) u+\frac{f''(0)}{2} u^2 , \quad \si(u)\approx \si(0)  , $$   
which yields first
\begin{align*}
u&\approx \si(0)\, \big(\cg\ast \dot{B}\big)+\bigg[f(0)\frac{t^2}{2}+f'(0) \ \big(\cg \ast \big[\1_{\R_+} u\big]\big)+\frac{f''(0)}{2} \big(\cg \ast \big[\1_{\R_+}  u^2\big]\big)\bigg]\, ,
\end{align*}
and by iterating the procedure, we formally get
\begin{align}
u&\approx \si(0)\, \big(\cg\ast \dot{B}\big)+f'(0)\si(0) \ \big(\cg \ast \big[\1_{\R_+}  \big(\cg\ast \dot{B}\big) \big]\big) +\frac{f''(0)}{2}\si(0)^2 \big(\cg \ast \big[\1_{\R_+}   \big(\cg\ast \dot{B}\big)^2\big]\big)+f(0)\frac{t^2}{2}\label{expansion-wave}\, .
\end{align}
To some extent, this local approximation of $u$ can be compared with the expansion in \eqref{expansion-sde}, and accordingly we are led to the following parallel between the fundamental elements in each formula, classified along their orders:
\begin{equation}\label{five-central-objects}
\int_0^. \dot{B}_s \longleftrightarrow \big(\cg\ast \dot{B},\cg \ast \big[\1_{\R_+}  \big(\cg\ast \dot{B}\big) \big]\big)=:(\<Psi>,\<IPsi>)\quad , \quad \int_0^. \int_0^s \dot{B}_r \otimes \dot{B}_s\longleftrightarrow \cg \ast \big[\1_{\R_+} \big(\cg\ast \dot{B}\big)^2\big]=:\<IPsi2>\, .
\end{equation}
The above symbols $\<Psi>,\<IPsi>,\<IPsi2>$ will be used for a better identification of these central processes, and in the same vein, we denote $\<Psi2>:=\big(\cg\ast \dot{B}\big)^2$ (this graphical convention is derived from the one used in the theory of regularity structures \cite{hai-14}). A particular common feature of the five random objects in \eqref{five-central-objects} is that their definitions do not depend on the (potential) solution $u$, and therefore these objects can be studied independently. 

\smallskip

Now, in the same spirit as in the above-described rough paths procedure (see again Proposition \ref{prop:intro-fbm}), the fundamental question is to determine whether the above elements $(\<Psi>,\<IPsi>,\<IPsi2>)$ do exist, in a space to be specified. Recall indeed that $\dot{B}$ is only defined as a general random distribution (see Definition \ref{defi:frac-noise}), and therefore its involvement within any \enquote{non-smooth} construction must be carefully justified. 

\

The existence issue for the first-order elements $(\<Psi>,\<IPsi>)$ has already been treated in \cite[Proposition 1.2]{De1}, and the result can be summed up as follows (see Section \ref{subsec:set-not} for the definition of the weighted Sobolev spaces $\ch_w^\al$).

\begin{proposition}[\cite{De1}]\label{prop:psi}
Fix $d\geq 1$ and let $\dot{B}$ be a space-time fractional noise  on $\R_+\times \R^d$, with index $H=(H_0,\ldots,H_d)\in (0,1)^{d+1}$. Set $H_+:=H_1+\ldots+H_d$. Consider a mollifier $\rho$ on $\R^{d+1}$, i.e. $\rho:\R^{d+1}\to \R_+$ is a smooth compactly-supported function of integral $1$, and for every $n\geq 0$, set $\rho_n(t,x):=2^{(d+1)n} \rho(2^n t,2^n x)$. Finally, for every $n\geq 0$, define
$$\dot{B}^n :=\rho_n \ast \dot{B}\, ,\quad  \<Psi>^n:=\cg \ast [\1_{\R_+}\dot{B}^n] \quad \text{and}\quad \<IPsi>^n:=\cg \ast \big[\1_{\R_+}  \<Psi>^n \big]\, .$$

\smallskip

Then, for every $\al>d-\frac12-(H_0+H_+)$ and every positive function $w\in L^1(\R^d)$, the sequence $(\<Psi>^n)$, resp. $(\<IPsi>^n)$, converges almost surely in the space $\cac(\R_+;\ch_w^{-\al})$, resp. $\cac(\R_+;\ch_w^{-\al+1})$.
\end{proposition}

\smallskip

The above result shows in particular that, as far as the construction of $(\<Psi>,\<IPsi>)$ is concerned, no restriction on the regularity of $\dot{B}$ (or more precisely its index $H$) is required. 

\smallskip

Going back to the developments \eqref{expansion-wave}-\eqref{five-central-objects}, we can henceforth concentrate on the existence issue for the second-order element $\<IPsi2>$ - just as we concentrated on the existence issue for the L{\'e}vy-area term in Proposition \ref{prop:intro-fbm}. And indeed, based on all the previous considerations, our investigations in the sequel will be motivated by the two following natural guidelines:

\com{

\textit{$(i)$ The possibility to \enquote{construct} the random object $\<IPsi2>=\cg\ast [\1_{\R_+}(\cg\ast \dot{B})^2]$ opens the way toward the interpretation and the well-posedness of equation \eqref{equa-intro}.}

\

\textit{$(ii)$ Conversely, if there is no (natural) interpretation for $\<IPsi2>$ - even as a general distribution -, then there is essentially no hope to find any (natural) interpretation for the model \eqref{equa-intro}.}

}

\

Of course, we do not pretend here that the sole existence of $\<IPsi2>$ (in situation $(i)$) immediately entails interpretation and well-posedness of equation \eqref{equa-intro} (see Remark \ref{rk:existence-not-sufficient} for further details). However, we claim that this existence is a necessary first step toward the treatment of the equation, and the impossibility to construct $\<IPsi2>$ fundamentally means ill-posedness for \eqref{equa-intro}.

\

Still keeping the pattern of Proposition \ref{prop:intro-fbm} in mind, we intend to exhibit a specific threshold for the sum $H_0+H_+$, where the situation switches from regime $(i)$ (existence of $\<IPsi2>$) to regime $(ii)$ (non-existence).

\smallskip

In this setting, observe first that, by the result of Proposition \ref{prop:psi}, the construction of $\<IPsi2>$ is straightforward whenever $H_0+H_+> d-\frac12$. Indeed, in this case, one has $\<Psi>\in \cac([0,T];\ch_w^\ka)$ for every $0<\ka< H_0+H_+-(d-\frac12)$. In particular, $\<Psi>$ is guaranteed to be a function of both the time and space variables, which immediately legitimates the interpretation of the square $\<Psi>^2$ as a basic pointwise product. From there, the interpretation of the convolution $\cg\ast [\1_{\R_+} \<Psi>^2]$, which gives birth to $\<IPsi2>$, becomes an easy task (see item $(i\text{-}a)$ of Theorem \ref{theo:threshold} for a precise statement in this configuration).

\

The situation becomes significantly more intricate when $H_0+H_+\leq d-\frac12$, since the linear solution $\<Psi>$ can only be understood as a general distribution with respect to the space variable. The interpretation of the square $(\<Psi>)^2$ must then go through a renormalization trick, implemented at the level of some approximation of $\<Psi>$. In the sequel, we will focus on Wick-type renormalization procedures: namely, given a smooth approximation $(\dot{B}^n)_{n\geq 0}$ of $\dot{B}$ and a \textit{deterministic} sequence of functions $(c_n:\R_+\times \R^d\to\R)_{n\geq 0}$, we set successively
\begin{equation}
\<Psi>^n:=\cg^n \ast [\1_{\R_+}\dot{B}^n]\, , \quad \<Psi2>^n(t,x):=(\<Psi>^n)^2(t,x)-c_n(t,x)\, , \quad \<IPsi2>^{(n)}:=\cg\ast \<Psi2>^n\, ,
\end{equation}
and then study the conditions on $(c_n)$ that could guarantee convergence of $\<IPsi2>^n$. 

\smallskip

The Wick-type renormalization method can be regarded as the most simple way to rescale an a priori diverging model $u$. In terms of local approximation (see \eqref{expansion-wave}), this deformation $u^n\to \hat{u}^n$ formally echoes as
\begin{align*}
\hat{u}^n&\approx \si(0)\, \<Psi>^n+f'(0)\si(0) \ \<IPsi>^n +\frac{f''(0)}{2}\si(0)^2\Big[ \<IPsi2>^n-\cg\ast c_n\Big]+f(0)\frac{t^2}{2}\, .    
\end{align*}

\

In the two-dimensional setting, that is when $d=2$ in \eqref{equa-intro}, first results on the success and the limits of the Wick renormalization method for $(\<Psi>)^2$ have already been obtained in \cite{De1,De2,gubi-koch-oh-1}. So far, the most complete statement in this direction can be found in \cite[Propositions 1.3 and 1.4]{De2} and summed up as follows (for the sake of conciseness, we refer the reader to \cite[Equation (4)]{De1} for the definition of the so-called harmonizable approximation of $\dot{B}$, close in spirit to the mollifying approximation).

\begin{proposition}[\cite{De2}]\label{prop:dime-2}
Fix $d=2$ and let $(\dot{B}^n)_{n\geq 0}$ be the harmonizable approximation of a fractional noise $\dot{B}$ on $\R_+\times \R^2$, with index $H=(H_0,H_1,H_2)\in (0,1)\times (0,\frac34)^2$. Set $H_+=H_1+H_2$ and for every $n\geq 0$, define successively 
\begin{equation}\label{ref-c-n}
\<Psi>^n:=\cg \ast [\1_{\R_+}\dot{B}^n]\, , \quad c_n(t,x):=\mathbb{E}\big[(\<Psi>^n)^2(t,x)\big] \quad \text{and}\quad \<Psi2>^n(t,x):=(\<Psi>^n)^2(t,x)-c_n(t,x) \, ,
\end{equation}
that is $\<Psi2>^n$ is derived from a standard Wick renormalization. In this setting:

\smallskip

\noindent
$(i)$ If $1<H_0+H_+\leq \frac32$, then for any (spatial) cut-off function $w:\R^2\to \R$, the sequence $(w\cdot \<Psi2>^n)$ converges almost surely in $\cac([0,T];\ch^{-2\al}(\R^2))$, for every $\al>\frac32-(H_0+H_1+H_2)$.

\smallskip

\noindent
$(ii)$ If $H_0+H_+\leq 1$, then for every $t>0$, every non-zero cut-off function $w:\R^2\to \R$ and \textit{every $\al\in \R$}, one has
$$\mathbb{E}\Big[\big\|w\cdot \<Psi2>^n(t,.)\big\|_{\ch^{-2\al}(\R^2)}^2\Big] \stackrel{n\to\infty}{\longrightarrow} \infty\, .$$ 
\end{proposition}

\

At first sight, the above assertions $(i)$-$(ii)$ indeed emphasizes a clear change of regime on the frontier $H_0+H_+=1$ (to be compared with the result of Proposition \ref{prop:intro-fbm} for the standard equation). Namely, when $H_0+H_+> 1$, we are able to interpret the (Wick) square $\<Psi2>$ as an element in the scale of spaces $\cac([0,T],\ch^{-2\al})$ ($\al\in \R$), whereas it is no longer possible to do so when $H_0+H_+\leq 1$. Nevertheless, in the general framework of the present study, the statement of Proposition \ref{prop:dime-2} still leaves several important questions in abeyance:

\

\noindent
$(\mathfrak{a})$ As we have just mentionned it, items $(i)$ and $(ii)$ of Proposition \ref{prop:dime-2} both rely on the scale of the functions of time with values in a Sobolev space $\ch^{-2\al}$, for some $\al\in \R$. In fact, this choice is mostly motivated by the convenient properties of the wave kernel $\cg$ in this scale of spaces, which, in \cite{De2}, pave the way toward a fixed-point argument for equation \eqref{equa-intro} (see Remark \ref{rk:existence-not-sufficient}). Putting this existence issue in a broader context, \textit{we may naturally wonder whether the explosion phenomenon in item $(ii)$ would still occur for a more general space-time distributions topology}. 

\smallskip

Consider the basic example provided by the noise $\dot{B}$ itself: we know that for every fixed $t\geq 0$, the approximation $x\mapsto \dot{B}^n_t(x)$ cannot converge as a space distribution (the noise is not a  function of time), whereas $(t,x)\mapsto \dot{B}^n_t(x)$ clearly converges to $\dot{B}$ as a space-time distribution. Based on the sole statement of Proposition \ref{prop:dime-2}, one cannot rule out the possibility of a similar behaviour for $\<Psi2>^n$. 

\

\noindent
$(\mathfrak{b})$ Perhaps more importantly: Proposition \ref{prop:dime-2} investigates the existence/non-existence issue for the (Wick) square $\<Psi2>$. And yet, as we have seen it in the above developments \eqref{expansion-wave}-\eqref{five-central-objects}, the central element at the core of the dynamics in \eqref{equa-intro} is not exactly $\<Psi2>$, but the convolution of $\<Psi2>$ with $\cg$, that is $\<IPsi2>$.  

It turns out that in several situations studied in the literature, the influence of convolution with a given (deterministic) kernel is not negligible when it comes to such existence issues. This can be basically observed on the transition from the noise $\dot{B}$ to the convolution $\<Psi>:=\cg\ast \dot{B}$. Indeed, while $\dot{B}$ can only be regarded as a general space-time distribution, it is possible to construct $\<Psi>$ as a function of time (we have seen it in Proposition \ref{prop:psi}). Many similar behaviours have been exhibited for the nonlinear heat equation on the torus: see for instance the contrast between Lemma 16 and Proposition 24 in \cite{e-jentzen-shen}, or between items (i) and (ii) in \cite[Proposition 1.9]{oh-okamoto}.

In light of these examples, it seems reasonable to think that convolution with $\cg$ might have an impact on the assertions of Proposition \ref{prop:dime-2}, and in particular \textit{we can legitimately question the exact value of the change-of-regime threshold when replacing $\<Psi2>^n$ with $\<IPsi2>^n:=\cg \ast \<Psi2>^n$}.

\

\noindent
$(\mathfrak{c})$ The results in Proposition \ref{prop:dime-2} lean on the standard Wick renormalization procedure, as it can be seen from the specific definition of $c_n$ in \eqref{ref-c-n}. Thus, based on this sole statement, it is not obvious to determine whether the explosion phenomenon in item $(ii)$ could be avoided with a more suitable (but still deterministic) renormalization sequence $c_n$, that is with a more general \textit{Wick-type} renormalization method.

\

\noindent
$(\mathfrak{d})$ Last but not least, the statement of Proposition \ref{prop:dime-2} is limited to the two-dimensional case, and to the harmonizable approximation of $\dot{B}$. \textit{It is then natural to ask if (and how) these considerations - including the above observations $(\mathfrak{a})$-$(\mathfrak{b})$-$(\mathfrak{c})$ - could be extended to any dimension $d\geq 3$, and to the more standard mollifying approximation of $\dot{B}$}.

\

For all these reasons, and also given the fundamental role of the process $\<IPsi2>$ in the dynamics of eq. \eqref{equa-intro}, we have found it important to search for a result that would both refine and generalize Proposition \ref{prop:dime-2}.

\section{Main result}

For a clear statement of our main contribution (Theorem \ref{theo:threshold} below), let us introduce a few definitions and notation first.

\subsection{Setting and notations}\label{subsec:set-not}
For any arbitrary domain $\Omega$ of a Euclidean space, we denote by $\cd(\Omega)$ the space of test-functions $\vp$ on $\Omega$, that is $\vp:\Omega\to \R$ is smooth and compactly supported.

For every function $\Psi\in L^2(\R\times \R^d)$, resp. $\psi\in L^2(\R^d)$, we denote the space-time, resp. space, Fourier transform by
$$\cf\Psi (\xi,\eta)=\int_{\R\times \R^d} dt dx \, e^{-\imath \xi t }e^{-\imath \langle \eta,x\rangle}\Psi(t,x)\, , \quad \text{resp.} \quad \cf_x\psi (\eta)=\int_{\R^d} dx \, e^{-\imath \langle \eta,x\rangle}\psi(x)\, .$$
We call a weight in $\R^d$ any strictly positive function $w\in L^1(\R^d)$. Given such a weight $w$, we define $L^p_w=L^p_w(\R^d)$ as the set of functions $f:\R^d \to \R$ for which
$$\|f\|_{L^p_w}^p:=\int_{\R} |f(x)|^p \, w(x)\, dx \ < \ \infty\, .$$
Then, for every $\al\in \R$, we define $\ch^{\al}_w=\ch^{\al}_w(\R^d)$ as the completion of $\cd(\R^d)$ with respect to the norm
\begin{equation}
\big\|u\big\|_{\ch^{\al}_w}:=\big\|\cf^{-1}\big( \{1+|.|^2\}^{\frac{\al}{2}}\cf u\big)\big\|_{L^2_w(\R^d)}\, .
\end{equation}

\

Let us now specify the definition of the fractional noise at the core of our investigations (this actually corresponds to the most widely used definition in the fractional literature).
\begin{definition}\label{defi:frac-noise}
Fix $H=(H_0,\ldots,H_d)\in (0,1)^{d+1}$. On a complete filtered probability space, we call a fractional noise on $\R^{d+1}$ with (Hurst) index $H$, and denote by $\dot{B}$, any centered Gaussian family
$$\big\{\dot{B}(\Phi), \ \Phi\in \cd(\R^{d+1})\big\}$$
with covariance given by the formula: for all $\Phi,\Psi\in \cd(\R^{d+1})$,
\begin{equation}\label{frac-noise-covariance}
\mathbb{E}\big[ \langle \dot{B},\Phi\rangle \, \langle \dot{B},\Psi\rangle \big]=\int_{\R\times\R^d} \mu_H(d\xi,d\eta)\,  \cf \Phi(\xi,\eta)  \overline{\cf \Psi(\xi,\eta)} , 
\end{equation}
where
\begin{equation}\label{density-noise}
\mu_H(d\xi,d\eta):=\frac{d\xi}{|\xi|^{2H_0-1}}\prod_{i=1,\ldots,d}\frac{d\eta_i}{|\eta_i|^{2H_i-1}}\, .
\end{equation}
\end{definition}

\

As we mentionned it earlier, our approximation of the noise $\dot{B}$, leading ultimately to the construction of $\<IPsi2>$, will be derived from a standard mollifying procedure.

\begin{definition}\label{defi:mollif}
We call a mollifier on $\R^{d+1}$ any smooth integrable function $\rho :\R^{d+1} \to \R_+$ such that $\mathcal F{\rho}$ is Lipschitz and $\int_{\R^{d+1}} \rho(s,x) \, ds dx=1$. 

\smallskip

Given a mollifier $\rho$ on $\R^{d+1}$, we define the mollifying sequence $(\rho_n)_{n\geq 0}$ by the formula: for all $\xi\in \R$ and $\eta\in \R^d$, $\rho_n(\xi,\eta):=2^{n(d+1)}\rho(2^n\xi,2^n \eta)$.

\end{definition}

\begin{remark}\label{remark on weights}
Following the above definition, any positive test-function $\rho\in \cd(\R^{d+1})$ with integral $1$ is a mollifier. The definition also covers the Gaussian-type mollifier used for instance in \cite[Section 3.2]{HHNT} or \cite[Section 5]{HN}, that is $\rho(s,x):=\1_{[0,1]}(s) p_1(x)$, where $p_1$ refers to the Gaussian density at time $1$.
\end{remark}

\

The notation $\cg$ for the wave kernel has already been introduced through \eqref{fourier-wave-kernel}. Recall that in dimension $d=1,2$, and for any fixed $t\geq 0$, this kernel $\cg_t$ is a well-defined function on $\R^d$, while for $d\geq 3$, $\cg_t$ can only be interpreted as a general space distribution. Therefore, for more rigour in our subsequent computations (where the kernel will be treated as a function), let us consider the following smooth approximation of $\cg_t$: for every $n \geq 1$, $t\in \R$ and $x\in \R^d$, we set 
\begin{equation}\label{defi-cg-n}
\cg^n(t,x):=\1_{\{t\geq 0\}}\int_{|\eta|\leq n}d\eta\, e^{\imath \langle x,\eta\rangle} \frac{\sin(t|\eta|)}{|\eta|} \, .
\end{equation}

\

\subsection{Main result}

We are now in a position to state our main contribution regarding existence/non-existence of the process $\<IPsi2>$ at the core of the dynamics of equation \eqref{equa-intro}.

\begin{theorem}\label{theo:threshold}
Fix a space dimension $d\geq 1$, as well as a weight $w$ on $\R^d$ and a mollifier $\rho$ on $\R^{d+1}$. On a complete filtered probability space, let $\dot{B}$ be a fractional noise of index $H\in (0,1)\times (0,\frac34)^{d}$, and for any given sequence of real functions $(c_n:\R_+\times \R^d\to\R)_{n\geq 0}$, define successively
\begin{equation}\label{defi:b-do-n}
\dot{B}^n :=\rho_n \ast \dot{B}\, ,\quad  \<Psi>^n:=\cg^n \ast [\1_{\R_+}\dot{B}^n]\, ,
\end{equation}
\begin{equation}\label{ref-c-n-theo}
\<Psi2>^n(t,x):=(\<Psi>^n)^2(t,x)-c_n(t,x) \quad \text{and} \quad \<IPsi2>^{(n)}:=\cg^n\ast [\1_{\R_+}\<Psi2>^n] \, .
\end{equation}

\

\noindent
In this setting, the following picture holds true (as before, we denote $H_+:=H_1+\ldots+H_d$):

\smallskip

\noindent
$(i\text{-}a)$ If $H_0+H_+>d-\frac12$, then by choosing $c_n(t,x):=0$ (i.e. no renormalization), the sequence $(\<IPsi2>^{(n)})_{n\geq 1}$ converges almost surely in the space $\cac(\R_+;\ch^{1+\ga}_{w})$, for every $0<\ga<H_0+H_+-(d-\frac12)$. 

\

\noindent
$(i\text{-}b)$ If $\frac{3d}{4}-\frac12 <H_0+H_+\leq d-\frac12$, then by considering the standard Wick renormalization procedure, that is by choosing
$$c_n(t,x):=\mathbb{E}\big[\<Psi>^n(t,x)^2\big]\, ,$$
the sequence $(\<IPsi2>^{(n)})_{n\geq 1}$ converges almost surely in the space $\cac(\R_+;\ch^{1-2\alpha}_{w})$, for every 
\begin{equation}\label{condition-alpha}
\al>\max\bigg(d-\frac12-(H_0+H_+),\frac{d-1}{4}\bigg)\, .
\end{equation}

\

\noindent
$(ii)$ If $H_0+H_+\leq \frac{3d}{4}-\frac12$, then there exists a non-empty subset $\mathcal{E}\subset \cd(\R_+\times \R^d)$ such that for any $\Phi\in \mathcal{E}$ and any sequence of real functions $(c_n:\R_+\times \R^d\to\R)_{n\geq 0}$, one has
\begin{equation}\label{divergence-theo}
\mathbb{E}\Big[\big|\langle\<IPsi2>^{(n)},\Phi\rangle\big|^2 \Big] \stackrel{n\to\infty}{\longrightarrow} \infty\, .
\end{equation}
\end{theorem}

\

The above properties $(i)$-$(ii)$ offer a striking illustration of the announced breakup phenomenon for $\<IPsi2>$ - and thus indirectly for the equation \eqref{equa-intro} itself -, on the frontier $H_0+H_+=\frac{3d}{4}-\frac12$: 

\smallskip

\noindent
$\bullet$ If $H_0+H_+>\frac{3d}{4}-\frac12$, then the process $\<IPsi2>$, defined as the limit of $\<IPsi2>^n$, can be interpreted as a function of time with values in a Sobolev space of (moderate) order $1-2\al \approx 1-\frac{d}{2}$.

\smallskip

\noindent
$\bullet$ If $H_0+H_+\leq \frac{3d}{4}-\frac12$, then a fundamental singularity issue prevents us from constructing $\<IPsi2>$, and accordingly any hope to analyze equation \eqref{equa-intro} essentially disappears in this rough situation.

\

In this way, Theorem \ref{theo:threshold} clearly corresponds to the exact counterpart of Proposition \ref{prop:intro-fbm} in the wave setting. When compared with the (pre-existing) result of Proposition \ref{prop:dime-2}, the statement also provides full answers to the four issues $\mathfrak{(a)}$-$\mathfrak{(b)}$-$\mathfrak{(c)}$-$\mathfrak{(d)}$ that we have addressed at the end of Section \ref{sec:intro}. Indeed:

\smallskip

\noindent
$\mathfrak{(a)}$ It points out that, in the situation where $H_0+H_+\leq \frac{3d}{4}-\frac12$, it is not possible to avoid the explosion phenonemon by extending the topological setting to more general space-time distributions, which shows how deep the singularity problem in this case.

\smallskip

\noindent
$\mathfrak{(b)}$ When $d=2$, it proves that the change-of-regime frontier $H_0+H_+=1$ for $\<Psi2>$ stays the same for $\<IPsi2>$, which sharply contrasts with similar situations in the heat setting. The observation would actually remain true for any $d\geq 1$.

\smallskip

\noindent
$\mathfrak{(c)}$ It confirms that the divergence in \eqref{divergence-theo} applies to any Wick-type renormalization procedure, that is to any sequence $(c_n)$, and not only to the standard one.

\smallskip

\noindent
$\mathfrak{(d)}$ It generalizes the results to any space dimension $d\geq 2$, and to the standard mollifying approximation.

\

Let us complete the statement of Theorem \ref{theo:threshold} with two additional comments.


\begin{remark}\label{rk:existence-not-sufficient}
Even if the existence of $\<IPsi2>$ clearly appears to us as a \textit{necessary} condition for a fruitful analysis of \eqref{equa-intro}, we are aware that the result of Theorem \ref{theo:threshold} (in the \enquote{open} situation $H_0+H_+>\frac{3d}{4}-\frac12$) is not \textit{sufficient} to derive interpretation and well-posedness of the equation.

\smallskip

In dimension $d=2$, a full treatment of the model \eqref{equa-intro} - with $f(u):=u^2$ and $\si(u):=1$ - can be found in \cite{De2}, under the same assumption $H_0+H_+>1$ (see also \cite{gubi-koch-oh-1} for a study in the specific white-noise situation). It turns out that, as $H_0+H_+$ gets close to $1$, not only does the analysis rely on $\<IPsi2>$, but it also requires the construction of an additional third-order processes $\<PsiIPsi2>$. Once endowed with those elements, the solution to \eqref{equa-intro} is obtained via a deterministic fixed-point argument, using the so-called Strichartz inequalities.

\smallskip

In dimension $d=3$, and as far as we know, the interesting rough regime $\frac74<H_0+H_+\leq \frac52$ has only been considered through the specific white-noise situation $H_0=\ldots=H_3=\frac12$, with also $f(u):=u^2$ and $\si(u):=1$. This analysis is the main topic of \cite{gubi-koch-oh-2}, where the authors appeal to the sophisticated machinery of paracontrolled calculus in order to set up a fixed-point procedure (note in particular the central role of $\<IPsi2>$ in \cite[Theorem 1.12]{gubi-koch-oh-2}). 

\smallskip
 
Finally, we are not aware of any full treatment of the model \eqref{equa-intro} for $d\geq 4$ and $\frac{3d}{4}-\frac12<H_0+H_+\leq d-\frac12$.

\end{remark}

\smallskip

\begin{remark}
In Theorem \ref{theo:threshold}, the condition $H_i<\frac34$ (for $i\geq 1$) guarantees in particular that the spatial density of the noise (given in \eqref{density-noise}) belongs to $L^2_{\text{loc}}(\R^d)$, which we will use in the construction of $\<IPsi2>$ (see e.g. the change of coordinates in \eqref{change-of-coord}). However, it should certainly be possible to handle the situation where $H_i>\frac34$ (for some $i\geq 1$) with similar arguments, at the price of an adaptation - and a less elegant formulation - of the threshold condition $H_0+H_+=\frac{3d}{4}-\frac12$.

\smallskip

On the other hand, the condition $H_i<\frac34$ still covers the most standard model considered in the literature, namely the white noise case $H_i=\frac12$ for $i=0,\ldots, d$. Note that in the latter configuration, the threshold condition of Theorem \ref{theo:threshold} reads as $\frac{d+1}{2}>\frac{3d}{4}-\frac12$, and accordingly, when $\dot{B}$ is a white noise, we cannot expect any Wick-type interpretation of equation \eqref{equa-intro} for $d\geq 4$.
\end{remark}

\

The rest of the paper is organized as follows. In Section \ref{subsec:cova} below, we briefly emphasize the covariance formulae of the processes $(\<Psi>^n,\<Psi2>^n)$, which will be the starting point of our analysis toward both items $(i)$ and $(ii)$ of Theorem \ref{theo:threshold}. Then Section \ref{sec:item-i} will be devoted to the proof of the existence part of the statement (item $(i)$), while Section \ref{sec:item-ii} will focus on the diverging property in the irregular regime (item $(ii)$). Finally, the appendix section contains further details about the divergence of the fractional L{\'e}vy area associated with \eqref{back-to-sde}, which allows for a direct comparison with the phenomenon in the wave setting.

\subsection{Covariance formulae}\label{subsec:cova}

\begin{lemma}\label{lem:cova}
In the setting of Theorem \ref{theo:threshold}, and for any fixed $n\geq 1$, let $\<Psi>^n$ and $\<Psi2>^n$ be the processes defined in \eqref{defi:b-do-n} and \eqref{ref-c-n-theo}, for a given (deterministic) function $c_n:\R_+\times \R^d \to\R$.

\smallskip

\noindent
$(1)$ For all $t,\tti\geq 0$ and $y,\yti\in \R^d$, it holds that
\begin{align}
&\mathbb{E}\Big[ \<Psi>^n(t,y) \overline{\<Psi>^n(\tti,\yti)} \Big]=\mathbb{E}\Big[ \<Psi>^n(t,y) \<Psi>^n(\tti,\yti) \Big]\nonumber\\
&=\int_{\R\times \R^d} \mu_{H}^{(n)}\!(d\xi,d\eta) \, e^{\imath \xi(t-\tti)} e^{\imath \langle \eta,y-\yti\rangle}\int_0^t ds \, e^{-\imath \xi s} \cf_x(\cg^n_s)(\eta)\int_0^{\tti} d\sti \, e^{\imath \xi \sti} \overline{\cf_x(\cg^n_{\sti})}(\eta)\, ,\label{cova-luxo}
\end{align}
where we have set
\begin{equation}\label{defi-mu-n}
\mu_{H}^{(n)}\!(d\xi,d\eta) :=\mu_H(d\xi,d\eta)\big|\cf \rho_n(-\xi,-\eta)\big|^2\, .
\end{equation}

\smallskip

\noindent
$(2)$ For all $t,\tti\geq 0$ and $y,\yti\in \R^d$, it holds that
\begin{align}\label{cova-cherry}
\mathbb{E}\Big[ \<Psi2>^n(t,y)\overline{\<Psi2>^n(\tti,\yti)}\Big]=2\Big(\mathbb{E}\Big[ \<Psi>^n(t,y) \overline{\<Psi>^n(\tti,\yti)} \Big]\Big)^2+\Big\{\mathbb{E}\Big[ \<Psi>^n(t,y)^2\Big]-c_n(t,y)\Big\}\Big\{\mathbb{E}\Big[ \overline{\<Psi>^n(\tti,\yti)}^2\Big]-\overline{c_n(\tti,\yti)}\Big\}\, .
\end{align}
\end{lemma}

\

\begin{proof}
By applying the covariance formula \eqref{frac-noise-covariance}, we get first that for all $s,\sti\geq 0$, $z,\zti\in \R^d$,
\begin{align*}
\mathbb{E}\Big[ \dot{B}^n(s,z)\dot{B}^n(\sti,\zti)\Big]&=\mathbb{E}\Big[ \langle\dot{B},\rho_n(s-.,z-.)\rangle \langle \dot{B},\rho_n(\sti-.,\zti-.)\rangle\Big]\\
&=\int_{\R\times \R^d} \mu_H(d\xi,d\eta)\, \big|\cf \rho_n(\xi,\eta)\big|^2 e^{-\imath \xi s}e^{\imath \xi \sti} e^{-\imath \langle \eta,z\rangle}e^{\imath \langle \eta,\zti\rangle}\\
&=\int_{\R\times \R^d} \mu^{(n)}_H\!(d\xi,d\eta)\,  e^{\imath \xi s}e^{-\imath \xi \sti} e^{\imath \langle \eta,z\rangle}e^{-\imath \langle \eta,\zti\rangle}\, .
\end{align*}
Therefore,
\begin{align*}
&\mathbb{E}\Big[ \<Psi>^n(t,y)\overline{\<Psi>^n(\tti,\yti)}\Big]=\mathbb{E}\Big[ \<Psi>^n(t,y) \<Psi>^n(\tti,\yti) \Big]\\
&=\int_0^t ds\int_0^{\tti} d\sti \int_{(\R^d)^2} dz d\zti\ \cg^n_{t-s}(y-z)\cg_{\tti-\sti}(\yti-\zti)\mathbb{E}\Big[ \dot{B}^n(s,z)\overline{\dot{B}^n(\sti,\zti)}\Big]\\
&=\int_0^t ds\int_0^{\tti} d\sti \int_{(\R^d)^2} dz d\zti\ \cg^n_{t-s}(y-z)\cg^n_{\tti-\sti}(\yti-\zti)\int_{\R\times \R^d} \mu_{H}^{(n)}\!(d\xi,d\eta)\, e^{\imath \xi s}e^{-\imath \xi \sti} e^{\imath \langle \eta,z\rangle}e^{\imath \langle \eta,\zti\rangle}\\
&=\int_{\R\times \R^d} \mu_{H}^{(n)}\!(d\xi,d\eta)\int_0^t ds\int_0^{\tti} d\sti \int_{(\R^d)^2} dz d\zti\ \cg^n_{t-s}(y-z)\cg^n_{\tti-\sti}(\yti-\zti) e^{\imath \xi s}e^{-\imath \xi \sti} e^{\imath \langle \eta,z\rangle}e^{-\imath \langle \eta,\zti\rangle}\\
&=\int_{\R\times \R^d} \mu_{H}^{(n)}\!(d\xi,d\eta) \, e^{\imath \xi(t-\tti)} e^{\imath \langle \eta,y-\yti\rangle}\int_0^t ds \, e^{-\imath \xi s} \cf_x(\cg^n_s)(\eta)\int_0^{\tti} d\sti \, e^{\imath \xi \sti} \overline{\cf_x(\cg^n_{\sti})}(\eta)\, ,
\end{align*}
which corresponds to the desired formula \eqref{cova-luxo}.

\smallskip

As for \eqref{cova-cherry}, it is a mere consequence of the Wick formula, which allows us to write
$$\mathbb{E}\Big[ \<Psi>^n(t,y)^2 \overline{\<Psi>^n(\tti,\yti)}^2 \Big]=2\, \mathbb{E}\Big[ \<Psi>^n(t,y) \overline{\<Psi>^n(\tti,\yti)} \Big]^2+\mathbb{E}\Big[ \<Psi>^n(t,y)^2\Big] \mathbb{E}\Big[\overline{\<Psi>^n(\tti,\yti)}^2 \Big]\, .$$
\end{proof}

\

\section{Above the threshold: proof of Theorem \ref{theo:threshold}, part $(i)$}\label{sec:item-i}

This section is devoted to the construction of the central element $\<IPsi2>$ in the two situations described in item $(i)$ of Theorem \ref{theo:threshold}, that is $H_0+H_+>d-\frac12$ (no renormalization) and $\frac{3d}{4}-\frac12 <H_0+H_+\leq d-\frac12$ (standard Wick renormalization).

\

To be more specific, in both cases, we will concentrate on the exhibition of a uniform bound (over $n$) for the second moment of the approximation $\<IPsi2>^n_t$, for fixed $t>0$ and relatively to a suitable topology in space - see the below statements of Propositions \ref{prop:moment-cherry-regu} and \ref{prop:moment-cherry}.

\

The transition from these uniform bounds to the assertions $(i)$-$(a)$ and $(i)$-$(b)$ of Theorem \ref{theo:threshold} is then a matter of standard arguments, that have been detailled many times in the literature. Such a procedure can be found for instance in \cite[Section 2.1]{De1} or \cite[Section 3]{De2} (see also \cite[Proposition 3.6]{mourrat-weber-xu} or \cite[Lemma 2.6]{oh-okamoto} for results in the same spirit). We leave it to the reader to check that our setting in this regard is not different from the one in the latter references. 

\

Let us finally report on a useful estimate related to the Fourier transform of $\cg^n$. This assertion is borrowed from \cite[Corollary 2.2]{De1}, and it will serve us several times in the sequel.
\begin{lemma}\label{lem:recov-h-0}
Fix $H_0\in (0,1)$ and $T>0$. Then for all $0\leq s\leq T$, $\eta\in \R^d$, $\ka \in \big[0,\min\big(H_0, \frac{1}{2}\big)\big)$ and $\varepsilon \in (0,\frac12-\ka)$, it holds that
$$\sup_{n\geq 1}\int_{\R}\frac{d\xi}{|\xi|^{2H_0-1}}\bigg|\int_0^s dr \, e^{-\imath \xi r} \cf_x(\cg^n_r)(\eta)\bigg|^2 \lesssim \frac{s^{2\ka}}{1+|\eta|^{1+2H_0-2\ka-2\varepsilon}} \ ,$$
where the proportional constant only depends on $T$. 
\end{lemma}

\

\subsection{Situation $(i)\text{-}(a)$}

\begin{proposition}\label{prop:moment-cherry-regu}
In the setting of Theorem \ref{theo:threshold}, assume that $H_0+H_+>d-\frac12$ and set $c_n(t,x):=0$ in the definitions introduced in \eqref{defi:b-do-n}-\eqref{ref-c-n-theo}.

\smallskip

Then for every $0<\ga <H_0+H_+-(d-\frac12)$, there exists $\ka>0$ such that for all $T>0$ and $t\in [0,T]$,
\begin{equation}\label{esti-mom-regul}
\sup_{n\geq 1}\mathbb{E}\Big[ \big\|  \<IPsi2>^n_t \big\|_{\ch^{1+\ga}_w(\R^d)}^2\Big] \lesssim \|w\|_{L^1(\R^d)} t^{2+4\ka} \, ,
\end{equation}
where the proportional constant only depends on $T$.
\end{proposition}

\begin{proof}
Fix $0<\ga <H_0+H_+-(d-\frac12)$. The moment under consideration can be readily expanded as
\begin{align}
&\mathbb{E}\Big[ \big\|  \<IPsi2>^n_t \big\|_{\ch^{1+\ga}_w(\R^d)}^2\Big]\nonumber\\
&=\int_{\R^d} dx \, w(x) \, \mathbb{E}\bigg[\bigg|\cf^{-1}\Big( \{1+|.|^2\}^{\frac{1+\ga}{2}}\int_0^t ds\, \cf_x \big(\cg^n_{t-s}\big)\cf_x \big(\<Psi2>^n_s\big)\Big)(x)\bigg|^2\bigg]\nonumber\\
&=\int_{\R^d} dx \, w(x) \, \mathbb{E}\bigg[\bigg|\int_{\R^d}d\la \, e^{\imath \langle x,\la\rangle} \{1+|\la|^2\}^{\frac12+\ga}\int_0^t ds\, \cf_x \big(\cg^n_{t-s}\big)(\la)\int_{\R^d} dy \, e^{-\imath \langle \la,y\rangle}\<Psi2>^n_s(y)\bigg|^2\bigg]\nonumber\\
&=\int_{\R^d} dx \, w(x) \int_{(\R^d)^2}d\la d\lati \, e^{\imath \langle x,\la-\lati\rangle} \{1+|\la|^2\}^{\frac{1+\ga}{2}}\{1+|\lati|^2\}^{\frac{1+\ga}{2}}\int_0^t ds\int_0^t d\sti\, \cf_x \big(\cg^n_{t-s}\big)(\la)\cf_x \big(\cg^n_{t-\sti}\big)(\lati)\nonumber\\
&\hspace{8cm}\int_{(\R^d)^2} dyd\yti  \, e^{-\imath \langle \la,y\rangle}e^{\imath \langle \lati,\yti\rangle}\mathbb{E}\Big[\<Psi2>^n_s(y) \overline{\<Psi2>^n_{\sti}(\yti)}\Big]\, .\label{expans-mom}
\end{align}
Now let us recall formula \eqref{cova-cherry}, which in this case becomes
\begin{align*}
\mathbb{E}\Big[ \<Psi2>^n(s,y)\overline{\<Psi2>^n(\sti,\yti)}\Big]=\mathbb{E}\Big[ \<Psi>^n(s,y)^2\Big]\mathbb{E}\Big[ \overline{\<Psi>^n(\sti,\yti)}^2\Big]+2\Big(\mathbb{E}\Big[ \<Psi>^n(s,y) \overline{\<Psi>^n(\sti,\yti)} \Big]\Big)^2\, .
\end{align*}
Injecting this decomposition into \eqref{expans-mom}, we deduce that
\begin{align}\label{decompo-mom-regu}
&\mathbb{E}\Big[ \big\|  \<IPsi2>^n_t \big\|_{\ch^{1+\ga}_w(\R^d)}^2\Big]=I^n_t+2 \, I\! I^n_t,
\end{align}
with
\begin{align*}
&I^n_t:=\int_{\R^d} dx \, w(x) \int_{(\R^d)^2}d\la d\lati \, e^{\imath \langle x,\la-\lati\rangle} \{1+|\la|^2\}^{\frac{1+\ga}{2}}\{1+|\lati|^2\}^{\frac{1+\ga}{2}}\\
&\hspace{0.5cm}\int_0^t ds\int_0^t d\sti\, \cf_x \big(\cg^n_{t-s}\big)(\la)\cf_x \big(\cg^n_{t-\sti}\big)(\lati)\int_{(\R^d)^2} dyd\yti  \, e^{-\imath \langle \la,y\rangle}e^{\imath \langle \lati,\yti\rangle}\mathbb{E}\Big[ \<Psi>^n(s,y)^2\Big]\mathbb{E}\Big[ \overline{\<Psi>^n(\sti,\yti)}^2\Big]\\
\end{align*}
and
\begin{align}
&I\! I^n_t:=\int_{\R^d} dx \, w(x) \int_{(\R^d)^2}d\la d\lati \, e^{\imath \langle x,\la-\lati\rangle} \{1+|\la|^2\}^{\frac{1+\ga}{2}}\{1+|\lati|^2\}^{\frac{1+\ga}{2}}\nonumber\\
&\hspace{1cm}\int_0^t ds\int_0^t d\sti\, \cf_x \big(\cg^n_{t-s}\big)(\la)\cf_x \big(\cg^n_{t-\sti}\big)(\lati)\int_{(\R^d)^2} dyd\yti  \, e^{-\imath \langle \la,y\rangle}e^{\imath \langle \lati,\yti\rangle}\Big(\mathbb{E}\Big[\<Psi>^n_s(y) \overline{\<Psi>^n_{\sti}(\yti)}\Big]\Big)^2\, .\label{defi-ii-n-t}
\end{align}

\

\noindent
\textit{Step 1: Estimation of $I^n_t$.}

\smallskip

Observe that $I^n_t$ can actually be recast as
\begin{align*}
I^n_t&=\int_{\R^d} dx \, w(x)\bigg| \int_{\R^d}d\la \, e^{\imath \langle x,\la\rangle} \{1+|\la|^2\}^{\frac{1+\ga}{2}}\int_0^t ds\, \cf_x \big(\cg^n_{t-s}\big)(\la)\int_{\R^d} dy  \, e^{-\imath \langle \la,y\rangle}\mathbb{E}\Big[ \<Psi>^n(s,y)^2\Big]\bigg|^2
\end{align*}
By applying the covariance formula \eqref{cova-luxo}, we immediately obtain the expression
\begin{align*}
&\mathbb{E}\Big[ \<Psi>^n(s,y)^2  \Big]=\int_{\R\times \R^d} \mu_{H}^{(n)}\!(d\xi,d\eta) \bigg| \int_0^s dr \, e^{-\imath \xi r} \cf_x(\cg^n_r)(\eta)\bigg|^2.
\end{align*}
Note that the latter quantity does not depend on the space variable $y$, and so, using the (formal) identity 
$$\int_{\R^d} dy  \, e^{-\imath \langle \la,y\rangle} =\delta_{\{\la=0\}},$$
where $\delta$ stands for the Dirac distribution, we get that
\begin{align*}
&\int_{\R^d}d\la \, e^{\imath \langle x,\la\rangle} \{1+|\la|^2\}^{\frac{1+\ga}{2}}\int_0^t ds\, \cf_x \big(\cg^n_{t-s}\big)(\la)\int_{\R^d} dy  \, e^{-\imath \langle \la,y\rangle}\mathbb{E}\Big[ \<Psi>^n(s,y)^2\Big]\\
&=\int_{\R\times \R^d} \mu_{H}^{(n)}\!(d\xi,d\eta)\int_0^t ds\, (t-s)\bigg| \int_0^s dr \, e^{-\imath \xi r} \cf_x(\cg^n_r)(\eta)\bigg|^2,
\end{align*}
which enables us to write $I^n_t$ as
\begin{align*}
I^n_t&=\|w\|_{L^1(\R^d)}\bigg| \int_{\R\times \R^d} \mu_{H}^{(n)}\!(d\xi,d\eta)\int_0^t ds\, (t-s)\bigg|\int_0^s dr \, e^{-\imath \xi r} \cf_x(\cg^n_r)(\eta)\bigg|^2\bigg|^2.
\end{align*}

\smallskip

At this point, observe that $\cf \rho_n(-\xi,-\eta)=\cf\rho(-2^{-n} \xi,-2^{-n}\eta)$, and thus, since $\rho$ is a mollifier, one has $\|\cf \rho_n\|_{L^\infty}\leq \|\cf \rho\|_{L^\infty}\leq \|\rho\|_{L^1}=1$, which immediately entails
\begin{equation}\label{bound-mu-n}
\mu_{H}^{(n)}\!(d\xi,d\eta)\leq \mu_{H}(d\xi,d\eta).
\end{equation}
Consequently,
\begin{align*}
&\bigg| \int_{\R\times \R^d} \mu_{H}^{(n)}\!(d\xi,d\eta)\int_0^t ds\, (t-s)\bigg|\int_0^s dr \, e^{-\imath \xi r} \cf_x(\cg^n_r)(\eta)\bigg|^2\bigg|\\
&\leq \int_{\R^d} \bigg(\prod_{i=1,\ldots,d} \frac{d\eta_i}{|\eta_i|^{2H_i-1}}\bigg)\int_0^t ds \, (t-s) \int_{\R}\frac{d\xi}{|\xi|^{2H_0-1}} \bigg| \int_0^s dr \, e^{-\imath \xi r} \cf_x(\cg^n_r)(\eta)\bigg|^2 ,
\end{align*}
and we can now use the estimate of Lemma \ref{lem:recov-h-0} to derive that for all $\ka,\varepsilon>0$ small enough,
\begin{align*}
I^n_t&\lesssim \|w\|_{L^1(\R^d)}\bigg|\int_{\R^d} \bigg(\prod_{i=1,\ldots,d} \frac{d\eta_i}{|\eta_i|^{2H_i-1}}\bigg)\frac{1}{1+|\eta|^{1+2H_0-2\ka-2\varepsilon}}\int_0^t ds \, (t-s)s^{2\ka}\bigg|^2\\
&\lesssim \|w\|_{L^1(\R^d)}t^{4+4\ka}\bigg|\int_{\R^d} \bigg(\prod_{i=1,\ldots,d} \frac{d\eta_i}{|\eta_i|^{2H_i-1}}\bigg)\frac{1}{1+|\eta|^{1+2H_0-2\ka-2\varepsilon}}\bigg|^2.
\end{align*}
Since $2H_i-1<1$, we are here allowed to implement a spherical change of coordinates and assert that for some finite proportional constant,
\begin{align*}
I^n_t
&\lesssim  \|w\|_{L^1(\R^d)}t^{4+4\ka}\bigg|\int_0^\infty \frac{dr }{r^{1-2(d-H_+)}\{1+r^{1+2H_0-2\ka-2\varepsilon}\}}\bigg|^2.
\end{align*}
Thanks to the assumption $H_0+H_+>d-\frac12$, it is easy to check that the integral in the latter bound is finite for all $\ka,\varepsilon>0$ small enough, and thus we have shown that
\begin{equation}\label{unif-bou-i-n-t}
\sup_{n\geq 1}\, I^n_t\lesssim \|w\|_{L^1(\R^d)} t^{4+4\ka}\, ,
\end{equation}
where the proportional constant only depends on $T$.

\

\noindent
\textit{Step 2: Estimation of $I\! I^n_t$.}

\smallskip

By using the covariance formula \eqref{cova-luxo}, we can write
\begin{align*}
&\Big(\mathbb{E}\Big[\<Psi>^n_s(y) \overline{\<Psi>^n_{\sti}(\yti)}\Big]\Big)^2\\
&=\bigg(\int_{\R\times \R^d} \mu_{H}^{(n)}\!(d\xi,d\eta) \, e^{\imath \xi(s-\sti)} e^{\imath \langle \eta,y-\yti\rangle}\int_0^s dr \, e^{-\imath \xi r} \cf_x(\cg^n_r)(\eta)\int_0^{\sti} d\rti \, e^{\imath \xi \rti} \overline{\cf_x(\cg^n_{\rti})}(\eta)\bigg)^2\\
&=\int_{\R\times \R^d} \mu_{H}^{(n)}\!(d\xi,d\eta)\int_{\R\times \R^d} \mu_{H}^{(n)}\!(d\xiti,d\etati) \, e^{\imath \xi(s-\sti)} e^{\imath \xiti(s-\sti)}e^{\imath \langle \eta,y-\yti\rangle}e^{\imath \langle \etati,y-\yti\rangle}\\
&\hspace{1cm}\int_0^s dr \, e^{-\imath \xi r} \cf_x(\cg^n_r)(\eta)\int_0^{\sti} d\rti \, e^{\imath \xi \rti} \overline{\cf_x(\cg^n_{\rti})}(\eta) \int_0^s dv \, e^{-\imath \xiti v} \cf_x(\cg^n_v)(\etati)\int_0^{\sti} d\vti \, e^{\imath \xiti \vti} \overline{\cf_x(\cg^n_{\vti})}(\etati).
\end{align*}
Injecting this expression into the definition of $I\! I^n_t$ and then using the (formal) identity
$$\int_{(\R^d)^2} dyd\yti  \, e^{-\imath \langle \la,y\rangle}e^{\imath \langle \lati,\yti\rangle}e^{\imath \langle \eta+\etati,y\rangle}e^{-\imath \langle \eta+\etati,\yti\rangle}=\delta_{\{\la=\eta+\etati\}}\delta_{\{\lati=\eta+\etati\}},$$
we get this time
\begin{align}
&I\! I^n_t=\|w\|_{L^1(\R^d)}\int_{\R\times \R^d} \mu_{H}^{(n)}\!(d\xi,d\eta)\int_{\R\times \R^d} \mu_{H}^{(n)}\!(d\xiti,d\etati) \, \{1+|\eta+\etati|^2\}^{1+\ga}\nonumber\\
&\hspace{1cm}\bigg|\int_0^t ds\, \cf_x \big(\cg^n_{t-s}\big)(\eta+\etati)e^{\imath s(\xi+\xiti)} \int_0^s dr \, e^{-\imath \xi r} \cf_x(\cg^n_r)(\eta) \int_0^s dv \, e^{-\imath \xiti v} \cf_x(\cg^n_v)(\etati)\bigg|^2.\label{expre-ii-n-t}
\end{align}
Let us recall the uniform bound \eqref{bound-mu-n}. Besides, given the definition \eqref{defi-cg-n} of $\cg^n$, it is clear that
\begin{equation}\label{estim-bas-four-g-n}
\big|\cf_x \big(\cg^n_{t-s}\big)(\eta+\etati)\big| \lesssim \frac{1+T}{1+|\eta+\etati|}\, ,
\end{equation}
where the proportional constant does not depend on $n$. By injecting these estimates into \eqref{expre-ii-n-t}, we deduce
\begin{align}
I\! I^n_t&\lesssim \|w\|_{L^1(\R^d)}\int_{\R\times \R^d} \mu_{H}(d\xi,d\eta)\int_{\R\times \R^d} \mu_{H}(d\xiti,d\etati) \, \{1+|\eta+\etati|^2\}^{\ga}\nonumber\\
&\hspace{3cm}t\int_0^t ds\, \bigg|\int_0^s dr \, e^{-\imath \xi r} \cf_x(\cg^n_r)(\eta)\bigg|^2\bigg| \int_0^s dv \, e^{-\imath \xiti v} \cf_x(\cg^n_v)(\etati)\bigg|^2\nonumber\\
&\lesssim \|w\|_{L^1(\R^d)}\, t\int_0^t ds\bigg[\int_{\R\times \R^d} \mu_{H}(d\xi,d\eta) \, \{1+|\eta|^2\}^{\ga} \bigg|\int_0^s dr \, e^{-\imath \xi r} \cf_x(\cg^n_r)(\eta)\bigg|^2\bigg]^2,\label{ii-n-t-inter}
\end{align}
where the proportional constant only depends on $T$, and where we have used the trivial inequality $\{1+|\eta+\etati|^2\}^{\ga}\lesssim \{1+|\eta|^2\}^{\ga}\{1+|\etati|^2\}^{\ga}$. Just as for $I^n_t$, we can now rely on the estimate of Lemma \ref{lem:recov-h-0} to assert that for $\ka,\varepsilon>0$ small enough,
\begin{align*}
&\int_{\R\times \R^d} \mu_{H}(d\xi,d\eta) \, \{1+|\eta|^2\}^{\ga} \bigg|\int_0^s dr \, e^{-\imath \xi r} \cf_x(\cg^n_r)(\eta)\bigg|^2\\
&=\int_{\R^d}\bigg(\prod_{i=1,\ldots,d} \frac{d\eta_i}{|\eta_i|^{2H_i-1}}\bigg)\{1+|\eta|^2\}^{\ga}\int_{\R}\frac{d\xi}{|\xi|^{2H_0-1}}\bigg|\int_0^s dr \, e^{-\imath \xi r} \cf_x(\cg^n_r)(\eta)\bigg|^2\\
&\lesssim s^{2\ka}\int_{\R^d}\bigg(\prod_{i=1,\ldots,d} \frac{d\eta_i}{|\eta_i|^{2H_i-1}}\bigg)\frac{1}{1+|\eta|^{1+2H_0-2\ga-2\ka-2\varepsilon}} \\
&\lesssim s^{2\ka}\int_0^\infty \frac{dr }{r^{1-2(d-2H_+)}\{1+r^{1+2H_0-2\ga-2\ka-2\varepsilon}\}}.
\end{align*}
Due to the assumption $0<\ga <H_0+H_+-(d-\frac12)$, we can guarantee that the latter integral is finite for all $\ka,\varepsilon >0$ small enough. Going back to \eqref{ii-n-t-inter}, we have proved that
\begin{equation}\label{unif-bou-ii-n-t}
\sup_{n\geq 1}\, I\! I^n_t\lesssim \|w\|_{L^1(\R^d)} t^{2+4\ka}\, ,
\end{equation}
where the proportional constant only depend on $T$.

\

Injecting \eqref{unif-bou-i-n-t} and \eqref{unif-bou-ii-n-t} into \eqref{decompo-mom-regu} finally yields the desired uniform estimate \eqref{esti-mom-regul}.

\end{proof}


\subsection{Situation $(i)$-$(b)$}

\begin{proposition}\label{prop:moment-cherry}
In the setting of Theorem \ref{theo:threshold}, assume that $\frac{3d}{4}-\frac12< H_0+H_+\leq d-\frac12$ and define the renormalizing sequence $(c_n)$ in \eqref{ref-c-n} as
$c_n(t,x):=\mathbb{E}\big[\<Psi>^n(t,x)^2\big]$.

\smallskip

Then for every $\al>0$ satisfying condition \eqref{condition-alpha}, there exists $\ka>0$ such that for all $T>0$ and $t\in [0,T]$,
\begin{equation}\label{estim-mom-cas-rug}
\sup_{n\geq 1}\mathbb{E}\Big[ \big\|  \<IPsi2>^n_t \big\|_{\ch^{1-2\al}_w(\R^d)}^2\Big] \lesssim \|w\|_{L^1(\R^d)} t^{2\ka+2} \, ,
\end{equation}
where the proportional constant only depends on $T$.
\end{proposition}

\begin{proof}

Note first that due to the assumption $H_0+H_+>\frac{3d}{4}-\frac12$, we can in fact reinforce condition \eqref{condition-alpha} and assume without loss of generality that 
\begin{equation}\label{condition-alpha-proof}
\max\bigg(d-\frac12-(H_0+H_+),\frac{d-1}{4}\bigg)<\al<\frac{d}{4}\, .
\end{equation}
With this condition in hand, we can start with expanding the considered moment as in the proof of Proposition \ref{prop:moment-cherry-regu}, that is
\begin{align*}
&\mathbb{E}\Big[ \big\|  \<IPsi2>^n_t \big\|_{\ch^{1-2\al}_w(\R^d)}^2\Big]\\
&=\int_{\R^d} dx \, w(x) \int_{(\R^d)^2}d\la d\lati \, e^{\imath \langle x,\la-\lati\rangle} \{1+|\la|^2\}^{\frac12-\al}\{1+|\lati|^2\}^{\frac12-\al}\int_0^t ds\int_0^t d\sti\, \cf_x \big(\cg^n_{t-s}\big)(\la)\cf_x \big(\cg^n_{t-\sti}\big)(\lati)\\
&\hspace{8cm}\int_{(\R^d)^2} dyd\yti  \, e^{-\imath \langle \la,y\rangle}e^{\imath \langle \lati,\yti\rangle}\mathbb{E}\Big[\<Psi2>^n_s(y) \overline{\<Psi2>^n_{\sti}(\yti)}\Big]\, .
\end{align*}
Now observe that formula \eqref{cova-cherry} here reduces to
\begin{align*}
\mathbb{E}\Big[ \<Psi2>^n(s,y)\overline{\<Psi2>^n(\sti,\yti)}\Big]=2\Big(\mathbb{E}\Big[ \<Psi>^n(s,y) \overline{\<Psi>^n(\sti,\yti)} \Big]\Big)^2\, .
\end{align*}
We can thus mimic the arguments ensuring the transition from \eqref{defi-ii-n-t} to \eqref{expre-ii-n-t}, and conclude that
\begin{align*}
\mathbb{E}\Big[ \big\|  \<IPsi2>^n_t \big\|_{\ch^{1-2\al}_w(\R^d)}^2\Big]&=2\|w\|_{L^1(\R^d)}\int_{\R\times \R^d} \mu_{H}^{(n)}\!(d\xi,d\eta)\int_{\R\times \R^d} \mu_{H}^{(n)}\!(d\xiti,d\etati) \, \{1+|\eta+\etati|^2\}^{1-2\al}\\
&\hspace{0.5cm}\bigg|\int_0^t ds\, \cf_x \big(\cg^n_{t-s}\big)(\eta+\etati)e^{\imath s(\xi+\xiti)} \int_0^s dr \, e^{-\imath \xi r} \cf_x(\cg^n_r)(\eta) \int_0^s dv \, e^{-\imath \xiti v} \cf_x(\cg^n_v)(\etati)\bigg|^2\, .
\end{align*}
Then, using both \eqref{bound-mu-n} and \eqref{estim-bas-four-g-n}, we deduce
\begin{align*}
&\mathbb{E}\Big[ \big\|  \<IPsi2>^n_t \big\|_{\ch^{1-2\al}_w(\R^d)}^2\Big]\lesssim \|w\|_{L^1(\R^d)}\, t\int_{(\R^d)^2} \frac{d\eta d\etati}{\{1+|\eta-\etati|^2\}^{2\al}} \bigg(\prod_{i=1,\ldots,d}\frac{1}{|\eta_i|^{2H_i-1}|\etati_i|^{2H_i-1}}\bigg) \\
&\hspace{2cm}\int_0^t ds\, \bigg[\int_{\R}\frac{d\xi}{|\xi|^{2H_0-1}}\bigg|\int_0^s dr \, e^{-\imath \xi r} \cf_x(\cg^n_r)(\eta)\bigg|^2\bigg]\bigg[\int_{\R}\frac{ d\xiti}{|\xiti|^{2H_0-1}}\bigg|\int_0^s dv \, e^{-\imath \xiti v} \cf_x(\cg^n_{v})(\etati)\bigg|^2\bigg] \, ,
\end{align*}
and we can apply the result of Lemma \ref{lem:recov-h-0} to obtain, for all $\ka,\varepsilon>0$ small enough,
\begin{align*}
\mathbb{E}\Big[ \big\|  \<IPsi2>^n_t \big\|_{\ch^{1-2\al}_w(\R^d)}^2\Big]&\lesssim t^{2\ka+2}\int_{(\R^d)^2}\frac{d\eta d\etati}{\{1+|\eta-\etati|^2\}^{2\al}} K_H(\eta) K_H(\etati),
\end{align*}
where the proportional constant only depends on $T$, and where we have set, for every $\eta\in \R^d$,
\begin{equation}\label{defi:k-h}
K_H(\eta)=K_{H,\ka,\varepsilon}(\eta):=\frac{1}{1+|\eta|^{1+2H_0-2\ka-2\varepsilon}}\prod_{i=1,\ldots,d}\frac{1}{|\eta_i|^{2H_i-1}}\, .
\end{equation}
The desired estimate \eqref{estim-mom-cas-rug} is now a straightforward consequence of the subsequent technical Lemma \ref{main-technical}, which achieves the proof of Proposition \ref{prop:moment-cherry}.
\end{proof}

\

\begin{lemma}\label{main-technical}
For all $H\in (0,1)^{d+1}$, $\ka,\varepsilon>0$ and $\eta\in \R^d$, let $K_H(\eta)=K_{H,\ka,\varepsilon}(\eta)$ be the quantity defined in \eqref{defi:k-h}.

\smallskip

If $\frac{3d}{4}-\frac12< H_0+H_+\leq d-\frac12$, then for every $\al>0$ satisfying \eqref{condition-alpha-proof}, there exists $\ka,\varepsilon>0$ small enough such that
$$\int_{\R^d\times \R^d} \frac{d\eta d\etati}{\{1+|\eta-\etati|^2\}^{2\alpha}} K_H(\eta) K_H(\etati) \ < \ \infty.$$
\end{lemma}

\

\begin{proof}[Proof of Lemma \ref{main-technical}]

Combining the identity $K_H(\pm\eta_1,\ldots,\pm \eta_d)=K_H(\eta)$ (for every $\eta\in \R^d$) with the elementary inequality $|\eta_i-\etati_i|\geq ||\eta_i|-|\etati_i||$ (for each $i=1,\ldots,d$), we can first write
\begin{align}
&\int_{\R^d\times \R^d} \frac{d\eta d\etati}{\{1+|\eta-\etati|^2\}^{2\alpha}} K_H(\eta) K_H(\etati)\nonumber\\
&\leq 4^d \int_{\R_+\times \R_+}d\eta_1d\etati_1 \cdots \int_{\R_+\times \R_+}d\eta_d d\etati_d \,  \frac{1}{\{1+|\eta-\etati|^2\}^{2\alpha}} K_H(\eta) K_H(\etati)\nonumber\\
&\leq 4^d \sum_{j_1,\ldots,j_d\in \{0,1\}}\int_{\mathcal{Q}_+^{(j_1)}}d\eta_1d\etati_1 \cdots \int_{\mathcal{Q}_+^{(j_d)}}d\eta_d d\etati_d \,  \frac{1}{\{1+|\eta-\etati|^2\}^{2\alpha}} K_H(\eta) K_H(\etati)\label{decompo-cq}
\end{align}
where we have set
$$\mathcal{Q}_+^{(0)}:=\big\{\eta_i,\etati_i >0: \ \frac12\eta_i \leq \etati_i \leq \frac32 \eta_i\big\} \quad \text{and} \quad \mathcal{Q}_+^{(1)}:=(\R_+\times \R_+)\backslash \mathcal{Q}_+^{(0)}.$$
In fact, for readily-checked symmetry reasons, we only need to focus on the generic situation where $j_1=...=j_k=0$ and $j_{k+1}=...=j_d=1$ in \eqref{decompo-cq}, for some $k\in \{0,\ldots,d\}$. In other words, from now on, we fix $k\in \{0,\ldots,d\}$ and try to show that the following integral is finite:
$$\cj_k:=\int_{\mathcal{Q}_+^{(0)}}d\eta_1d\etati_1 \cdots \int_{\mathcal{Q}_+^{(0)}}d\eta_k d\etati_k\int_{\mathcal{Q}_+^{(1)}}d\eta_{k+1}d\etati_{k+1}\cdots\int_{\mathcal{Q}_+^{(1)}}d\eta_d d\etati_d \,  \frac{1}{\{1+|\eta-\etati|^2\}^{2\alpha}} K_H(\eta) K_H(\etati)\, .$$

\

\

\noindent
\textbf{First case:} $k=0$. For $(\eta_i,\etati_i)\in \mathcal{Q}_+^{(1)}$, one has $|\eta_i-\etati_i|\geq \frac13\max( |\eta_i|,|\etati_i|)$, which leads us to
$$\cj_0=\int_{\mathcal{Q}_+^{(1)}}d\eta_{1}d\etati_{1}\cdots\int_{\mathcal{Q}_+^{(1)}}d\eta_d d\etati_d \,  \frac{1}{\{1+|\eta-\etati|^2\}^{2\alpha}} K_H(\eta) K_H(\etati)\lesssim \bigg[\int_{\R_+^d}\frac{d\eta}{\{1+|\eta|^2\}^{\alpha}} K_H(\eta)\bigg]^2\, .$$
Then, going back to the expression \eqref{defi:k-h} of $K_H$, we get that
\begin{align*}
\int_{\R_+^d}\frac{d\eta}{\{1+|\eta|^2\}^{\alpha}} K_H(\eta)&\lesssim \int_{\R_+^d}\frac{d\eta}{1+|\eta|^{2\al+1+2H_0-2\ka-2\varepsilon}}\prod_{i=1,\ldots,d}\frac{1}{|\eta_i|^{2H_i-1}}\\
&\lesssim \int_0^\infty \frac{dr}{r^{1-2(d-H_+)}\{1+r^{2p}\}}\, ,
\end{align*}
where $p:=\al+\frac12+H_0-\ka-\varepsilon$. Since we have assumed $\al>d-\frac12-(H_0+H_+)$, we can pick $\ka,\varepsilon>0$ small enough so that $1-2(d-H_+)+2p>1$, which achieves to prove that $\cj_0$ is finite.

\

\noindent
\textbf{Second case:} $k\in \{1,\ldots,d\}$. For the sake of clarity, let us rely on the following standard convention.

\

\begin{notation}\label{nota:prod-vec}
Given $\eta \in \R^d$ and $1\leq i\leq j\leq d$, we set $\eta_{i::j}:=(\eta_i,\ldots,\eta_j)\in \R^{j-i+1}$. Besides, for $v,\tilde{v}\in \R^k$, we denote the componentwise product by $v*_.\tilde{v}:=(v_1\tilde{v}_1,\ldots,v_k\tilde{v}_k)\in \R^k$.
\end{notation}

\

With this convention in hand, observe first that for all $\eta,\etati \in \R_+^d$ and $\beta\in \big[\frac12,\frac32\big]^k$, one has
$$K_H(\eta_{1::k}*_.\beta,\etati_{(k+1)::d})\lesssim K_H(\eta_{1::k},\etati_{(k+1)::d})\, .$$
Therefore, by performing the elementary change of variables $\beta_\ell=\frac{\etati_\ell}{\eta_\ell}$ for $\ell=1,\ldots,k$, we deduce that
\begin{align*}
&\cj_k\lesssim \int_{\R_+^k} d\eta_1 \cdots d\eta_k\,  \big\{\eta_1\cdots \eta_k\big\}\int_{\mathcal{Q}_+^{(1)}}d\eta_{k+1}d\etati_{k+1}\cdots\int_{\mathcal{Q}_+^{(1)}}d\eta_d d\etati_d \,  K_H(\eta) K_H(\eta_{1::k},\etati_{(k+1)::d}) \\
&\hspace{6cm} \int_{[\frac12,\frac32]^k}  \frac{d\beta}{\{1+|\eta_{1::k}*_.(1-\beta)|^2+|\eta_{(k+1)::d}-\etati_{(k+1)::d}|^2\}^{2\alpha}} \, .
\end{align*}
Now, just as in the first situation, recall that $|\eta_i-\etati_i|\geq \frac13\max( |\eta_i|,|\etati_i|)$ when $(\eta_i,\etati_i)\in \mathcal{Q}_+^{(1)}$, so that
\begin{align*}
&\cj_k\lesssim \int_{\R_+^k} d\eta_1 \cdots d\eta_k\,  \big\{\eta_1\cdots \eta_k\big\}\int_{\mathcal{Q}_+^{(1)}}d\eta_{k+1}d\etati_{k+1}\cdots\int_{\mathcal{Q}_+^{(1)}}d\eta_d d\etati_d \,  K_H(\eta) K_H(\eta_{1::k},\etati_{(k+1)::d}) \\
&\hspace{2cm} \int_{[\frac12,\frac32]^k}d\beta\,  \frac{1}{\{1+|\eta_{1::k}*_.(1-\beta)|^2+|\eta_{(k+1)::d}|^2\}^{\alpha}} \frac{1}{\{1+|\eta_{1::k}*_.(1-\beta)|^2+|\etati_{(k+1)::d}|^2\}^{\alpha}}\\
&\lesssim \int_{\R_+^k} d\eta_1 \cdots d\eta_k\, \big\{\eta_1\cdots \eta_k\big\}\int_{\R_+^{d-k}}d\eta_{k+1}\cdots d\eta_d\int_{\R_+^{d-k}}d\etati_{k+1}\cdots  d\etati_d \,  K_H(\eta_{1::k},\eta_{(k+1)::d}) K_H(\eta_{1::k},\etati_{(k+1)::d}) \\
&\hspace{2cm} \int_{[\frac12,\frac32]^k}d\beta\,  \frac{1}{\{1+|\eta_{1::k}*_.(1-\beta)|^2+|\eta_{(k+1)::d}|^2\}^{\alpha}} \frac{1}{\{1+|\eta_{1::k}*_.(1-\beta)|^2+|\etati_{(k+1)::d}|^2\}^{\alpha}},
\end{align*}
where the second inequality only consists of an extension of the integration domain for $\eta_j,\etati_j$, $j\geq k+1$.

\smallskip

Applying the Cauchy-Schwarz inequality (with respect to $\beta$), we obtain
\begin{align}
&\cj_k\nonumber\\
&\lesssim \int_{\R_+^k} d\eta_1 \cdots d\eta_k\, \big\{\eta_1\cdots \eta_k\big\}\int_{\R_+^{d-k}}d\eta_{k+1}\cdots d\eta_d\int_{\R_+^{d-k}}d\etati_{k+1}\cdots  d\etati_d \,  K_H(\eta_{1::k},\eta_{(k+1)::d}) K_H(\eta_{1::k},\etati_{(k+1)::d})\nonumber \\
& \bigg(\int_{[\frac12,\frac32]^k} \frac{d\beta}{\{1+|\eta_{1::k}*_.(1-\beta)|^2+|\eta_{(k+1)::d}|^2\}^{2\alpha}}\bigg)^{\frac12} \bigg(\int_{[\frac12,\frac32]^k} \frac{d\beta}{\{1+|\eta_{1::k}*_.(1-\beta)|^2+|\etati_{(k+1)::d}|^2\}^{2\alpha}}\bigg)^{\frac12}\nonumber\\
&\lesssim \int_{\R_+^k} d\eta_1 \cdots d\eta_k\, \big\{\eta_1\cdots \eta_k\big\}\nonumber\\
&\hspace{1cm}\bigg[\int_{\R_+^{d-k}}d\eta_{k+1}\cdots d\eta_d \,  K_H(\eta_{1::k},\eta_{(k+1)::d}) \bigg(\int_{[\frac12,\frac32]^k} \frac{d\beta}{\{1+|\eta_{1::k}*_.(1-\beta)|^2+|\eta_{(k+1)::d}|^2\}^{2\alpha}}\bigg)^{\frac12}\bigg]^2\nonumber\\
&\lesssim \int_{\R_+^d} d\eta \, \big\{\eta_1\cdots \eta_k\big\}K_H(\eta)^2\int_{\R_+^{d-k}}d\la \int_{[\frac12,\frac32]^k} \frac{d\beta}{\{1+|\eta_{1::k}*_.(1-\beta)|^2+|\la|^2\}^{2\alpha}}\, ,\label{cj-k-intermed}
\end{align}
where we have used again Cauchy-Schwarz inequality (with respect to $\eta_{(k+1)::d}$) to get the third estimate.

\

Thanks to condition \eqref{condition-alpha-proof}, and since $k\in \{1,\ldots,d\}$, we can assert that for any $\varepsilon >0$ small enough,
$$0<2\alpha-\frac{d-k}{2}-\frac{\varepsilon}{2}<\frac{k}{2}\, . $$
Therefore, setting $a_k:=\frac{1}{k}\big(2\alpha-\frac{d-k}{2}-\frac{\varepsilon}{2}\big)$, one has $0<a_k<\frac12$. With these observations and notation in mind, let us go back to \eqref{cj-k-intermed} and write first
\begin{align}
\cj_k&\lesssim \int_{\R_+^d} d\eta \, \big\{\eta_1\cdots \eta_k\big\}K_H(\eta)^2\int_{\R_+^{d-k}}\frac{d\la}{\{1+|\la|^2\}^{\frac{d-k}{2}+\frac{\varepsilon}{2}}}\int_{[\frac12,\frac32]^k} \frac{d\beta}{\{1+|\eta_{1::k}*_.(1-\beta)|^2\}^{2\alpha-\frac{d-k}{2}-\frac{\varepsilon}{2}}}\nonumber\\
&\lesssim \int_{\{|\eta|\leq 1\}} d\eta \, K_H(\eta)^2+\int_{\{|\eta|\geq 1\}} d\eta \, \big|\eta_1\cdots \eta_k\big|K_H(\eta)^2\int_{[\frac12,\frac32]^k} \frac{d\beta}{\{1+|\eta_{1::k}*_.(1-\beta)|^2\}^{2\alpha-\frac{d-k}{2}-\frac{\varepsilon}{2}}}\, .\label{cj-k-intermed-2}
\end{align}
Since we have assumed $0<H_i<\frac34$, it is readily checked that $K_H\in L^2_{\text{loc}}(\R^d)$, and so the first integral in \eqref{cj-k-intermed-2} is indeed finite. As for the second integral, one has, with the above-introduced parameter $a_k\in (0,\frac12)$,
\begin{align*}
&\int_{\{|\eta|\geq 1\}} d\eta \, \big|\eta_1\cdots \eta_k\big|K_H(\eta)^2\int_{[\frac12,\frac32]^k} \frac{d\beta}{\{1+|\eta_{1::k}*_.(1-\beta)|^2\}^{2\alpha-\frac{d-k}{2}-\frac{\varepsilon}{2}}}\\
&\lesssim \int_{\{|\eta|\geq 1\}} d\eta \, \bigg\{\frac{1}{|\eta_1|^{2a_k-1}}\cdots \frac{1}{|\eta_k|^{2a_k-1}}\bigg\}K_H(\eta)^2\bigg(\int_{[\frac12,\frac32]} \frac{d\beta}{|1-\beta|^{2a_k}}\bigg)^k\\
&\lesssim \int_{\{|\eta|\geq 1\}} d\eta \, \bigg\{\frac{1}{|\eta_1|^{2a_k-1}}\cdots \frac{1}{|\eta_k|^{2a_k-1}}\bigg\}K_H(\eta)^2\\
&\lesssim \int_{\{|\eta|\geq 1\}} \frac{d\eta}{|\eta|^{2+4H_0-4\ka-4\varepsilon}} \, \bigg\{\frac{1}{|\eta_1|^{4H_1-3+2a_k}}\cdots \frac{1}{|\eta_k|^{4H_k-3+2a_k}}\bigg\}\bigg\{\frac{1}{|\eta_{k+1}|^{4H_{k+1}-2}}\cdots \frac{1}{|\eta_d|^{4H_d-2}}\bigg\}\, .
\end{align*}
Recall again that $0<H_j<\frac34$ and $0<2a_k<1$, which ensures that $4H_j-3+2a_k<1$ for $j=1,\ldots,k$ and $4H_j-2<1$ for $j=k+1,\ldots,d$. We are therefore in a position to apply a spherical change of variables and assert that, for some finite proportional constant,
\begin{align}\label{change-of-coord}
&\int_{\{|\eta|\geq 1\}} \frac{d\eta}{|\eta|^{2+4H_0-4\ka-4\varepsilon}} \, \bigg\{\frac{1}{|\eta_1|^{4H_1-3+2a_k}}\cdots \frac{1}{|\eta_k|^{4H_k-3+2a_k}}\bigg\}\bigg\{\frac{1}{|\eta_{k+1}|^{4H_{k+1}-2}}\cdots \frac{1}{|\eta_d|^{4H_d-2}}\bigg\}\lesssim \int_1^\infty \frac{dr}{r^{p}}\, ,
\end{align}
where
\begin{align*}
p&:=(1-d)+(2+4H_0-4\ka-4\varepsilon)+\sum_{j=1}^k (4H_j-3+2a_k)+\sum_{j=k+1}^d (4H_j-2)\\
&=3-d+4(H_0+H_+)-3k-2(d-k)+2ka_k-4\ka-4\varepsilon\\
&=3-3d+4(H_0+H_+)-k+\big[4\alpha-(d-k)-\varepsilon\big]-4\ka-4\varepsilon\\
&=3+4\Big[\al-d+(H_0+H_+)\big]-4\ka-5\varepsilon \, .
\end{align*}
Since $\al>d-\frac12-(H_0+H_+)$, we can pick $\ka,\varepsilon>0$ small enough so that $p>1$. Going back to \eqref{cj-k-intermed-2}, this achieves to show that $\cj_k$ is finite for $k\in \{1,\ldots,d\}$, and thus the proof of Lemma \ref{main-technical} is now complete.

\end{proof}

\section{Below the threshold: proof of Theorem \ref{theo:threshold}, part $(ii)$}\label{sec:item-ii}

We now turn to the proof of part $(ii)$ of Theorem \ref{theo:threshold}, and to this end, fix $H\in (0,1)^{d+1}$ such that $H_0+H_+\leq \frac{3d}{4}-\frac12 $ for the whole section.

\

The set $\ce \subset \cd(\R_+\times \R^d)$ of test-functions $\Phi$ that will serve us to establish the divergence property \eqref{divergence-theo} can be immediately introduced as follows.

\begin{definition}\label{defi-space-ce}
We define $\ce\subset \cd(\R_+\times \R^d)$ as the set of test-functions $\Phi(t,x):=\vp(t)\psi(x)$, where $\vp\in \cd(\R_+)$ and $\psi\in \cd(\R^d)$ satisfy:

\smallskip

\noindent
$(1)$ $\vp\geq 0$, $\text{supp}\ \vp\subset (0,1)$ and $\int_{\frac34}^{1} dt \, \vp(t) >0$.

\smallskip

\noindent
$(2)$ $\cf_x(\psi)(\eta)=\cf_x(\psi)(-\eta)$ for every $\eta\in \R^d$, and $\inf_{|\theta|\leq 1}|\cf_x(\psi)(\theta)|^2>0$.
\end{definition}

\smallskip

It is clear that $\ce\neq \emptyset$. From now on, we fix such a test-function $\Phi=\vp\otimes \psi\in \ce$, as well as an arbitrary sequence $c_n:\R_+\times \R^d\to \R$ in the definition \eqref{ref-c-n-theo} of $\<IPsi>^n$.

\

For every fixed $n\geq 1$, the moment under consideration in \eqref{divergence-theo} can be expanded as
\begin{align*}
&\mathbb{E}\Big[ \big|\langle \<IPsi2>^n,\Phi\rangle \big|^2\Big] = \int_{(0,\infty)^2} dt d\tti\int_{(\R^d)^2} dyd\yti \, \Phi(t,y)\overline{\Phi(\tti,\yti)} \mathbb{E} \Big[\<IPsi2>^n(t,y) \overline{\<IPsi2>^n(\tti,\yti)}\Big]\\
&=\int_{(0,\infty)^2} dt d\tti\int_{(\R^d)^2} dyd\yti \, \Phi(t,y)\overline{\Phi(\tti,\yti)}\int_0^t du \int_0^{\tti}d\uti \int_{(\R^d)^2} dz d\zti \, \cg^n_{t-u}(y-z)\cg^n_{\tti-\uti}(\yti-\zti) \mathbb{E} \Big[\<Psi2>^n_u(z) \overline{\<Psi2>^n_{\uti}(\zti)}\Big],
\end{align*}
which, owing to formula \eqref{cova-cherry}, yields
\begin{align*}
&\mathbb{E}\Big[ \big|\langle \<IPsi2>^n,\Phi\rangle \big|^2\Big]\\
&=\int_{(0,\infty)^2} dt d\tti\int_{(\R^d)^2} dyd\yti \, \Phi(t,y)\overline{\Phi(\tti,\yti)}\int_0^t du \int_0^{\tti}d\uti  \int_{(\R^d)^2} dz d\zti \, \cg^n_{t-u}(y-z)\cg^n_{\tti-\uti}(\yti-\zti) \\
&\hspace{4cm}\Big[2\Big(\mathbb{E}\Big[ \<Psi>^n(u,z) \overline{\<Psi>^n(\uti,\zti)} \Big]\Big)^2+\Big\{\mathbb{E}\Big[ \<Psi>^n(u,z)^2\Big]-c_n(u,z)\Big\}\Big\{\mathbb{E}\Big[ \overline{\<Psi>^n(\uti,\zti)}^2\Big]-\overline{c_n(\uti,\zti)}\Big\}\Big]\\
&=2\int_{(0,\infty)^2} dt d\tti\int_{(\R^d)^2} dyd\yti \, \Phi(t,y)\overline{\Phi(\tti,\yti)}\int_0^t du \int_0^{\tti}d\uti  \int_{(\R^d)^2} dz d\zti \, \cg^n_{t-u}(y-z)\cg^n_{\tti-\uti}(\yti-\zti) \Big(\mathbb{E}\Big[ \<Psi>^n(t,y) \overline{\<Psi>^n(\uti,\zti)} \Big]\Big)^2\\
&\hspace{1cm}+\bigg|\int_{(0,\infty)} dt \int_{\R^d} dy \, \Phi(t,y)\int_0^t du   \int_{\R^d} dz  \, \cg^n_{t-u}(y-z) \Big\{\mathbb{E}\Big[ \<Psi>^n(u,z)^2\Big]-c_n(u,z)\Big\}\bigg|^2\, .
\end{align*}
As a result, it holds that
\begin{equation}\label{fir-low-bo}
\mathbb{E}\Big[ \big|\langle \<IPsi2>^n,\Phi\rangle \big|^2\Big]\geq 2 \, \ci_n \, ,
\end{equation}
where
\begin{align*}
&\ci_n:=\\
&\int_{(0,\infty)^2} dt d\tti\int_{(\R^d)^2} dyd\yti \, \Phi(t,y)\overline{\Phi(\tti,\yti)}\int_0^t du \int_0^{\tti}d\uti  \int_{(\R^d)^2} dz d\zti \, \cg^n_{t-u}(y-z)\cg^n_{\tti-\uti}(\yti-\zti) \Big(\mathbb{E}\Big[ \<Psi>^n(t,y) \overline{\<Psi>^n(\uti,\zti)} \Big]\Big)^2.
\end{align*}
Our objective now is to show that $\ci_n \to \infty$ as $n\to \infty$, which, in particular, emphasizes the fact that the sequence $(c_n)$ has no influence on the divergence phenomenon.

\smallskip

Let us apply the covariance formula \eqref{cova-luxo} to expand $\ci_n$ as
\begin{align*}
\ci_n &=\int_{(0,\infty)^2} dt d\tti\int_{(\R^d)^2} dyd\yti \, \Phi(t,y)\overline{\Phi(\tti,\yti)}\int_0^t du \int_0^{\tti}d\uti \int_{(\R^d)^2} dz d\zti \, \cg^n_{t-u}(y-z)\cg^n_{\tti-\uti}(\yti-\zti) \\
&\hspace{3cm}\int_{\R\times \R^d} \mu_{H}^{(n)}\!(d\xi,d\eta) \, e^{\imath \xi(u-\uti)} e^{\imath \langle \eta,z-\zti\rangle}\int_0^u ds \, e^{-\imath \xi s} \cf_x(\cg^n_s)(\eta)\int_0^{\uti} d\sti \, e^{\imath \xi \sti} \overline{\cf_x(\cg^n_{\sti})}(\eta)\\
&\hspace{4cm}\int_{\R\times \R^d} \mu_{H}^{(n)}\!(d\xiti,d\etati) \,e^{\imath \xiti(u-\uti)} e^{\imath \langle \etati,z-\zti\rangle}\int_0^u dr \, e^{-\imath \xiti r} \cf_x(\cg^n_r)(\etati)\int_0^{\uti} d\rti \, e^{\imath \xiti \rti} \overline{\cf_x(\cg^n_{\rti})}(\etati)\\
&=\int_{\R\times \R^d} \mu_{H}^{(n)}\!(d\xi,d\eta) \int_{\R\times \R^d} \mu_{H}^{(n)}\!(d\xiti,d\etati) \\
&\hspace{0.2cm}\bigg|\int_0^\infty dt\int_0^t du \int_{(\R^d)^2} dy dz \, \Phi(t,y) \cg^n_{t-u}(y-z) e^{\imath u(\xi+\xiti)} e^{\imath \langle z,\eta+\etati\rangle} \int_0^u ds \, e^{-\imath \xi s} \cf_x(\cg^n_s)(\eta) \int_0^u dr \, e^{-\imath \xiti r}\cf_x(\cg^n_r)(\etati)\bigg|^2\, .
\end{align*}
In this expression, we can rewrite the integral over $t$ as
\begin{align*}
&\int_0^\infty dt\, \vp(t)\int_0^t du\,  e^{\imath u(\xi+\xiti)}\int_{\R^d} dy\, \psi(y) \int_{\R^d} dz \,  \cg^n_{t-u}(y-z)  e^{\imath \langle z,\eta+\etati\rangle} \int_0^u ds \, e^{-\imath \xi s} \cf_x(\cg^n_s)(\eta) \int_0^u dr \, e^{-\imath \xiti r}\cf_x(\cg^n_r)(\etati)\\
&=\int_0^\infty dt\, \vp(t)\int_0^t du \, e^{\imath u(\xi+\xiti)}\int_{\R^d}dy\, \psi(y)e^{\imath \langle y,\eta+\etati\rangle} \\
&\hspace{4cm}\int_{\R^d} dz \,  \cg^n_{t-u}(z)  e^{-\imath \langle z,\eta+\etati\rangle} \int_0^u ds \, e^{-\imath \xi s} \cf_x(\cg^n_s)(\eta) \int_0^u dr \, e^{-\imath \xiti r}\cf_x(\cg^n_r)(\etati)\\
&=\cf_x(\psi)(\eta+\etati)\int_0^\infty dt \, \vp(t)\int_0^t du \, e^{\imath u(\xi+\xiti)}\cf_x(\cg^n_{t-u})(\eta+\etati)\int_0^u ds \, e^{-\imath \xi s} \cf_x(\cg^n_s)(\eta) \int_0^u dr \, e^{-\imath \xiti r}\cf_x(\cg^n_r)(\etati)\, .
\end{align*}
Therefore, recalling also the definition \eqref{defi-mu-n} of $\mu_{H}^{(n)}$, we get
\begin{align}
&\ci_n=\int_{\R\times \R^d} \mu_{H}^{(n)}\!(d\xi,d\eta) \int_{\R\times \R^d} \mu_{H}^{(n)}\!(-d\xiti,-d\etati)\,  \big|\cf_x(\psi)(\eta-\etati)\big|^2\nonumber\\
&\hspace{2cm}\bigg|\int_0^\infty dt \, \vp(t)\int_0^t du \, e^{\imath u(\xi-\xiti)}\cf_x(\cg^n_{t-u})(\eta-\etati)\int_0^u ds \, e^{-\imath \xi s} \cf_x(\cg^n_s)(\eta) \int_0^u dr \, e^{\imath \xiti r}\cf_x(\cg^n_r)(\etati)\bigg|^2\nonumber\\
&\geq \int_{(1,\infty)^d\cap \cac_d} \bigg(\prod_{i=1,\ldots,d}\frac{d\eta_i}{|\eta_i|^{2H_i-1}}\bigg)\int_{\eta_1}^{2\eta_1}\frac{d\etati_1}{|\etati_1|^{2H_1-1}}\cdots\int_{\eta_d}^{2\eta_d}\frac{d\etati_d}{|\etati_d|^{2H_d-1}} \big|\cf_x(\psi)(\eta-\etati)\big|^2\nonumber\\
&\hspace{1cm}\int_{|\eta|}^{2|\eta|} \frac{d\xi}{|\xi|^{2H_0-1}}\int_{|\etati|}^{2|\etati|} \frac{d\xiti}{|\xiti|^{2H_0-1}}\big|\cf \rho_n(-\xi,-\eta)\big|^2\big|\cf \rho_n(\xiti,\etati)\big|^2\nonumber\\
&\hspace{1.5cm}\bigg|\int_0^\infty dt \, \vp(t)\int_0^t du \, e^{\imath u(\xi-\xiti)}\cf_x(\cg^n_{t-u})(\eta-\etati)\int_0^u ds \, e^{-\imath \xi s} \cf_x(\cg^n_s)(\eta) \int_0^u dr \, e^{\imath \xiti r}\cf_x(\cg^n_r)(\etati)\bigg|^2 \, ,\label{first-low-bou}
\end{align}
where the set $\cac_d$ is defined through the spherical-coordinates expression
\begin{equation}\label{defi:ens-c-d}
\cac_d:=\bigg\{\Big(r \cos(\theta_1),r\sin(\theta_1)\cos(\theta_2),r\sin(\theta_1)\sin(\theta_2)\cos(\theta_3),\ldots,r\prod_{i=1}^{d-1}\sin(\theta_i)\Big); \ r>0,\ \theta\in \Big[\frac{\pi}{8},\frac{3\pi}{8}\Big]^{d-1}\bigg\}\, .
\end{equation}
At this point, observe that due to $H_0+H_+\leq \frac{3d}{4}-\frac12$, we can find a new set parameters $H_0',\ldots,H_d'\in (0,1)$ such that 
\begin{equation}\label{h-prime}
H_i'\geq H_i \ (i=0,\ldots,d) \quad \text{and} \quad H_0'+H_+'=\frac{3d}{4}-\frac12, \quad  \text{where} \ H_+':=\sum_{i=1}^d H_i'\, .
\end{equation}
Once endowed with these parameters, and going back to the above estimate \eqref{first-low-bou}, we easily see that
\begin{equation}\label{second-bo}
\ci_n\geq \cj_n\, ,
\end{equation}
with
\begin{align*}
\cj_n&:=\int_{(1,\infty)^d\cap \cac_d} \bigg(\prod_{i=1,\ldots,d}\frac{d\eta_i}{|\eta_i|^{2H'_i-1}}\bigg)\int_{\eta_1}^{2\eta_1}\frac{d\etati_1}{|\etati_1|^{2H'_1-1}}\cdots\int_{\eta_d}^{2\eta_d}\frac{d\etati_d}{|\etati_d|^{2H'_d-1}} \big|\cf_x(\psi)(\eta-\etati)\big|^2\\
&\hspace{1cm}\int_{|\eta|}^{2|\eta|} \frac{d\xi}{|\xi|^{2H'_0-1}}\int_{|\etati|}^{2|\etati|} \frac{d\xiti}{|\xiti|^{2H'_0-1}}\big|\cf \rho_n(-\xi,-\eta)\big|^2\big|\cf \rho_n(\xiti,\etati)\big|^2\\
&\hspace{1.5cm}\bigg|\int_0^\infty dt \, \vp(t)\int_0^t du \, e^{\imath u(\xi-\xiti)}\cf_x(\cg^n_{t-u})(\eta-\etati)\int_0^u ds \, e^{-\imath \xi s} \cf_x(\cg^n_s)(\eta) \int_0^u dr \, e^{\imath \xiti r}\cf_x(\cg^n_r)(\etati)\bigg|^2\, .
\end{align*}
The rest of the proof is devoted to showing that $\cj_n \to \infty$ as $n\to \infty$, and to this end, we shall lean on the following decomposition of the integral over $t$.
\begin{lemma}\label{lem:decompo-m-r}
For all $\xi,\xiti\geq 0$ and $\eta,\etati\in \R^d$, one has
\begin{align}
&\int_0^\infty dt \, \vp(t)\int_0^t du \, e^{\imath u(\xi-\xiti)}\cf_x(\cg^n_{t-u})(\eta-\etati)\int_0^u ds \, e^{-\imath \xi s} \cf_x(\cg^n_s)(\eta) \int_0^u dr \, e^{\imath \xiti r}\cf_x(\cg^n_r)(\etati)\nonumber\\
&\hspace{4cm}=\1_{\{|\eta|\leq n\}}\1_{\{|\etati|\leq n\}} \Big[M^n_\vp((\xi,\eta),(\xiti,\etati))+R^n_\vp((\xi,\eta),(\xiti,\etati))\Big]\, ,\label{expans-int-t}
\end{align}
with
\begin{align}
&M^n_\vp((\xi,\eta),(\xiti,\etati))\nonumber\\
&:=\frac{1}{4|\eta| |\etati|}\int_0^\infty dt \, \vp(t)\int_0^t du \, e^{\imath u(\xi-\xiti)}\cf_x(\cg^n_{t-u})(\eta-\etati)\int_0^u ds \, e^{-\imath s(\xi-|\eta|) }  \int_0^u dr \, e^{\imath  r(\xiti-|\etati|)}\label{defi-m-n}
\end{align}
and where, for every $\ka\in [0,1]$,
\begin{equation}\label{bound-r-n}
\sup_{n\geq 1}\big|R^n_\vp((\xi,\eta),(\xiti,\etati))\big|\lesssim \frac{1}{|\eta||\etati|}\bigg[\frac{1}{|\xi-|\eta||^\ka |\xiti+|\etati||}+\frac{1}{|\xi+|\eta|| |\xiti+|\etati||}+\frac{1}{|\xi+|\eta|| |\xiti-|\etati||^\ka}\bigg] \, .
\end{equation}

\end{lemma}

\begin{proof}[Proof  of Lemma \ref{lem:decompo-m-r}]
By writing
$$\cf_x(\cg^n_s)(\eta)=\frac{1}{2\imath |\eta|} \big\{e^{\imath s|\eta|}-e^{-\imath s |\eta|}\big\}\1_{\{|\eta|\leq n\}}\, ,$$
the integral under consideration can indeed be decomposed as a sum $M^n_\vp+R^n_\vp$, with $M^n_\vp$ given by \eqref{defi-m-n} and
\begin{align*}
&R^n_\vp((\xi,\eta),(\xiti,\etati)):=-\1_{\{|\eta|\leq n\}}\1_{\{|\etati|\leq n\}}\frac{1}{4|\eta| |\etati|}\int_0^\infty dt \, \vp(t)\int_0^t du \, e^{\imath u(\xi-\xiti)}\cf_x(\cg^n_{t-u})(\eta-\etati)\\
&\hspace{5cm}\int_0^u ds \int_0^u dr \, e^{-\imath s\xi }  e^{\imath  r\xiti}\big\{e^{\imath s|\eta|}e^{\imath r|\etati|}-e^{-\imath s |\eta|}e^{\imath r|\etati|}+e^{-\imath s |\eta|}e^{-\imath r |\etati|}\big\}\\
&=-\1_{\{|\eta|\leq n\}}\1_{\{|\etati|\leq n\}}\frac{1}{4|\eta| |\etati|}\int_0^1 dt \, \vp(t)\int_0^t du \, e^{\imath u(\xi-\xiti)}\cf_x(\cg^n_{t-u})(\eta-\etati)\\
&\hspace{0.1cm}\bigg\{\int_0^u ds \, e^{-\imath s(\xi-|\eta|) }  \int_0^u dr\, e^{\imath  r(\xiti+|\etati|)}-\int_0^u ds\, e^{-\imath s(\xi+|\eta|) }  \int_0^u dr  \,  e^{\imath  r(\xiti+|\etati|)}+\int_0^u ds\, e^{-\imath s(\xi+|\eta|)} \int_0^u dr  \,   e^{\imath  r(\xiti-|\etati|)}\bigg\}\, ,
\end{align*}
where we have used the assumption $\text{supp}\, \vp\subset (0,1)$. Based on the latter expansion, and since 
$$\sup_{0\leq t\leq 1} |\cf_x(\cg^n_{t})(\eta-\etati) |\leq 1\, ,$$
we get
\begin{align*}
&\big|R^n_\vp((\xi,\eta),(\xiti,\etati))\big|\lesssim \frac{1}{|\eta| |\etati|}\|\vp\|_\infty \int_0^1 du \, \bigg\{\bigg|\int_0^u ds \, e^{-\imath s(\xi-|\eta|) }  \int_0^u dr\, e^{\imath  r(\xiti+|\etati|)}\bigg|\\
&\hspace{3cm}+\bigg|\int_0^u ds\, e^{-\imath s(\xi+|\eta|) }  \int_0^u dr  \,  e^{\imath  r(\xiti+|\etati|)}\bigg|+\bigg|\int_0^u ds\, e^{-\imath s(\xi+|\eta|)} \int_0^u dr  \,   e^{\imath  r(\xiti-|\etati|)}\bigg|\bigg\}\, .
\end{align*}
The bound \eqref{bound-r-n} immediately follows.
\end{proof}

Let us go back to \eqref{second-bo}, and with expansion \eqref{expans-int-t} in hand, decompose $\cj_n$ into
\begin{equation}\label{decompo-cj-n}
\cj_n:=\cj_M^n+\cj_R^n+\cj_{M,R}^n\, ,
\end{equation}
with
\begin{align}
&\cj^n_M:=\int_{(1,\infty)^d\cap \cac_d} \bigg(\prod_{i=1,\ldots,d}\frac{d\eta_i}{|\eta_i|^{2H'_i-1}}\bigg)\int_{\eta_1}^{2\eta_1}\frac{d\etati_1}{|\etati_1|^{2H'_1-1}}\cdots\int_{\eta_d}^{2\eta_d}\frac{d\etati_d}{|\etati_d|^{2H'_d-1}}\1_{\{|\eta|\leq n\}}\1_{\{|\etati|\leq n\}} \nonumber\\
&\hspace{1.5cm}\big|\cf_x(\psi)(\eta-\etati)\big|^2\int_{|\eta|}^{2|\eta|} \frac{d\xi}{|\xi|^{2H'_0-1}}\int_{|\etati|}^{2|\etati|} \frac{d\xiti}{|\xiti|^{2H'_0-1}}\big|\cf \rho_n(-\xi,-\eta)\big|^2\big|\cf \rho_n(\xiti,\etati)\big|^2\big|M^n_\vp((\xi,\eta),(\xiti,\etati))\big|^2\, ,\label{defi:i-n-m}
\end{align}
\begin{align}
&\cj^n_R:=\int_{(1,\infty)^d\cap \cac_d} \bigg(\prod_{i=1,\ldots,d}\frac{d\eta_i}{|\eta_i|^{2H'_i-1}}\bigg)\int_{\eta_1}^{2\eta_1}\frac{d\etati_1}{|\etati_1|^{2H'_1-1}}\cdots\int_{\eta_d}^{2\eta_d}\frac{d\etati_d}{|\etati_d|^{2H'_d-1}}\1_{\{|\eta|\leq n\}}\1_{\{|\etati|\leq n\}} \nonumber\\
&\hspace{1cm}\big|\cf_x(\psi)(\eta-\etati)\big|^2\int_{|\eta|}^{2|\eta|} \frac{d\xi}{|\xi|^{2H'_0-1}}\int_{|\etati|}^{2|\etati|} \frac{d\xiti}{|\xiti|^{2H'_0-1}}\, \big|\cf \rho_n(-\xi,-\eta)\big|^2\big|\cf \rho_n(\xiti,\etati)\big|^2\big|R^n_\vp((\xi,\eta),(\xiti,\etati))\big|^2\, ,\label{defi:i-n-r}
\end{align}
and
\begin{align*}
&\cj^n_{M,R}:=\int_{(1,\infty)^d\cap \cac_d} \bigg(\prod_{i=1,\ldots,d}\frac{d\eta_i}{|\eta_i|^{2H'_i-1}}\bigg)\int_{\eta_1}^{2\eta_1}\frac{d\etati_1}{|\etati_1|^{2H'_1-1}}\cdots\int_{\eta_d}^{2\eta_d}\frac{d\etati_d}{|\etati_d|^{2H'_d-1}} \1_{\{|\eta|\leq n\}}\1_{\{|\etati|\leq n\}}\\
&\hspace{2cm}\big|\cf_x(\psi)(\eta-\etati)\big|^2\int_{|\eta|}^{2|\eta|} \frac{d\xi}{|\xi|^{2H'_0-1}}\int_{|\etati|}^{2|\etati|} \frac{d\xiti}{|\xiti|^{2H'_0-1}}\big|\cf \rho_n(-\xi,-\eta)\big|^2\big|\cf \rho_n(\xiti,\etati)\big|^2\\
&\hspace{3cm}\big\{M^n_\vp((\xi,\eta),(\xiti,\etati))\overline{R^n_\vp((\xi,\eta),(\xiti,\etati))}+\overline{M^n_\vp((\xi,\eta),(\xiti,\etati))}R^n_\vp((\xi,\eta),(\xiti,\etati))\big\}\, .
\end{align*}

\smallskip

We will now establish that $\cj^n_M\to \infty$ as $n\to \infty$, while $\cj^n_R$ and $\cj^n_{M,R}$ are uniformly bounded over $n$.

\

\subsection{Analysis of $\cj^n_M$}

One has $\cf \rho_n(\xi,\eta)=\cf\rho(2^{-n} \xi,2^{-n}\eta)$ and, following Definition \ref{defi:mollif}, we know that $\cf\rho(0)=1$. Therefore, for every $n$ large enough and for all $(\xi,\eta),(\xi,\etati)$ such that $1\leq |\eta|\leq n,1\leq |\etati|\leq n,|\eta|<\xi<2|\eta|$ and $|\etati|<\xiti<2|\etati|$, one has
$$\big|\cf \rho_n(-\xi,-\eta)\big|\geq \frac12 \quad  \text{and} \quad \big|\cf \rho_n(\xiti,\etati)\big| \geq \frac12\, .$$
Based on this uniform lower bound, we get that for all $\eta,\etati$ with $1\leq |\eta|\leq n$ and $1\leq |\etati|\leq n$,
\begin{align}
&\int_{|\eta|}^{2|\eta|} \frac{d\xi}{|\xi|^{2H'_0-1}}\int_{|\etati|}^{2|\etati|} \frac{d\xiti}{|\xiti|^{2H'_0-1}}\big|\cf \rho_n(-\xi,-\eta)\big|^2\big|\cf \rho_n(-\xiti,\etati)\big|^2\big|M^n_\vp((\xi,\eta),(\xiti,\etati))\big|^2\nonumber\\
&\geq \frac{1}{16} \int_{|\eta|}^{2|\eta|} \frac{d\xi}{|\xi|^{2H'_0-1}}\int_{|\etati|}^{2|\etati|} \frac{d\xiti}{|\xiti|^{2H'_0-1}}\big|M^n_\vp((\xi,\eta),(\xiti,\etati))\big|^2,\label{i-m-n-1}
\end{align}
and then, using only elementary changes of variables,
\begin{align}
&\int_{|\eta|}^{2|\eta|} \frac{d\xi}{|\xi|^{2H'_0-1}}\int_{|\etati|}^{2|\etati|} \frac{d\xiti}{|\xiti|^{2H'_0-1}}\big|M^n_\vp((\xi,\eta),(\xiti,\etati))\big|^2\nonumber\\
&= \frac{1}{|\eta|^{2H'_0-2} |\etati|^{2H'_0-2}}\int_{1}^{2} \frac{d\xi}{|\xi|^{2H'_0-1}}\int_{1}^{2} \frac{d\xiti}{|\xiti|^{2H'_0-1}}\big|M^n_\vp((\xi|\eta|,\eta),(\xiti|\etati|,\etati))\big|^2\nonumber\\
&\geq\frac{c}{|\eta|^{2H'_0} |\etati|^{2H'_0}} \int_{1}^{2} d\xi\int_{1}^{2} d\xiti\, \bigg|\int_0^\infty dt \, \vp(t)\int_0^t du \, e^{\imath u(\xi|\eta|-\xiti|\etati|)}\cf_x(\cg^n_{t-u})(\eta-\etati)\int_0^u ds \, e^{\imath  s|\eta|(\xi-1)} \int_0^u dr \, e^{-\imath r|\etati|(\xiti-1)}\bigg|^2\nonumber\\
&\geq\frac{c}{|\eta|^{2H'_0} |\etati|^{2H'_0}} \int_{0}^{1} d\xi\int_{0}^{1} d\xiti\, \bigg|\int_0^\infty dt \, \vp(t)\int_0^t du \, e^{\imath u(\xi|\eta|-\xiti|\etati|)}e^{\imath u(|\eta|-|\etati|)}\cf_x(\cg^n_{t-u})(\eta-\etati)\int_0^u ds \, e^{\imath  s|\eta|\xi} \int_0^u dr \, e^{-\imath r|\etati|\xiti}\bigg|^2\nonumber\\
&\geq\frac{c}{|\eta|^{2H'_0+1} |\etati|^{2H'_0+1}} \int_{0}^{|\eta|} d\xi\int_{0}^{|\etati|} d\xiti\, \bigg|\int_0^\infty dt \, \vp(t)\int_0^t du \, e^{\imath u(\xi-\xiti)}e^{\imath u(|\eta|-|\etati|)}\cf_x(\cg^n_{t-u})(\eta-\etati)\int_0^u ds \, e^{\imath  s\xi} \int_0^u dr \, e^{-\imath r\xiti}\bigg|^2\nonumber\\
&\geq\frac{c}{|\eta|^{2H'_0+1} |\etati|^{2H'_0+1}} \int_{0}^{1} d\xi\int_{0}^{1} d\xiti\, \bigg|\int_0^\infty dt \, \vp(t)\int_0^t du \,  e^{\imath u(|\eta|-|\etati|)}\cf_x(\cg^n_{t-u})(\eta-\etati)\int_u^{2u} ds \, e^{\imath  s\xi} \int_u^{2u} dr \, e^{-\imath r\xiti}\bigg|^2.\label{i-m-n-2}
\end{align}
By injecting \eqref{i-m-n-1} and \eqref{i-m-n-2} into \eqref{defi:i-n-m}, we obtain
\begin{align*}
&\cj^n_M\geq c\int_{(1,\infty)^d\cap \cac_d} \bigg(\prod_{i=1,\ldots,d}\frac{d\eta_i}{|\eta_i|^{2H'_i-1}}\bigg)\int_{\eta_1}^{2\eta_1}\frac{d\etati_1}{|\etati_1|^{2H'_1-1}}\cdots\int_{\eta_d}^{2\eta_d}\frac{d\etati_d}{|\etati_d|^{2H'_d-1}} \big|\cf_x(\psi)(\eta-\etati)\big|^2 \1_{\{|\eta|\leq n\}}\1_{\{|\etati|\leq n\}}\\
&\hspace{1.5cm}\frac{1}{|\eta|^{2H'_0+1} |\etati|^{2H'_0+1}} \int_{0}^{1} d\xi\int_{0}^{1} d\xiti\, \bigg|\int_0^\infty dt \, \vp(t)\int_0^t du \, e^{\imath u(|\eta|-|\etati|)}\cf_x(\cg^n_{t-u})(\eta-\etati)\int_u^{2u} ds \, e^{\imath  s\xi} \int_u^{2u} dr \, e^{-\imath r\xiti}\bigg|^2\,
\end{align*}
Let us now proceed with the change of variable $\be_i=\frac{\etati_i}{\eta_i}$, $i=1,\ldots,d$. Using the convention introduced in Notation \ref{nota:prod-vec}, one has clearly $|\be *_. \eta|\leq 2|\eta|$ and $\1_{\{|\beta*_. \eta|\leq n\}}\geq \1_{\{|\eta|\leq \frac{n}{2}\}}$ for every $\beta\in [1,2]^d$, which gives
\begin{align*}
&\cj^n_M\geq c \int_{(1,\infty)^d\cap \cac_d} \bigg(\prod_{i=1,\ldots,d}\frac{d\eta_i}{|\eta_i|^{4H'_i-3}}\bigg)\frac{1}{|\eta|^{4H'_0+2}}\1_{\{|\eta|\leq \frac{n}{2}\}}\int_{[1,2]^d}d\beta \,\big|\cf_x(\psi)(\eta*_.(1-\be))\big|^2\\
&\hspace{1.5cm}\int_{0}^{1} d\xi\int_{0}^{1} d\xiti\, \bigg|\int_0^\infty dt \, \vp(t)\int_0^t du \, e^{\imath u(|\eta|-|\eta *_. \beta|)}\cf_x(\cg^n_{t-u})(\eta*_.(1-\be))\int_u^{2u} ds \, e^{\imath  s\xi} \int_u^{2u} dr \, e^{-\imath r\xiti}\bigg|^2\, ,
\end{align*}
and then, since $\cf_x(\psi)(\eta)=\cf_x(\psi)(-\eta)$,
\begin{align}
&\cj^n_M\nonumber\\
&\geq c \int_{(1,\infty)^d\cap \cac_d} \bigg(\prod_{i=1,\ldots,d}\frac{d\eta_i}{|\eta_i|^{4H'_i-3}}\bigg)\frac{1}{|\eta|^{4H'_0+2}}\1_{\{|\eta|\leq \frac{n}{2}\}}\int_{[0,1]^d}d\beta\, \big|\cf_x(\psi)(\eta*_. \be)\big|^2\nonumber\\
&\hspace{1cm} \int_{0}^{1} d\xi\int_{0}^{1} d\xiti\, \bigg|\int_0^\infty dt \, \vp(t)\int_0^t du \, e^{\imath u(|\eta|-|\eta*_. (\beta-1)|)}\cf_x(\cg^n_{t-u})(\eta*_.\beta)\int_u^{2u} ds \, e^{\imath  s\xi} \int_u^{2u} dr \, e^{-\imath r\xiti}\bigg|^2\nonumber\\
&\geq c \int_{(1,\infty)^d\cap \cac_d} \bigg(\prod_{i=1,\ldots,d}\frac{d\eta_i}{|\eta_i|^{4H'_i-2}}\bigg)\frac{1}{|\eta|^{4H'_0+2}}\1_{\{|\eta|\leq \frac{n}{2}\}}\int_{0}^{\eta_1}d\theta_1\cdots\int_{0}^{\eta_d}d\theta_d\, \big|\cf_x(\psi)(\theta_1,\ldots,\theta_d)\big|^2\nonumber\\
&\hspace{1cm} \int_{0}^{1} d\xi\int_{0}^{1} d\xiti\, \bigg|\int_0^\infty dt \, \vp(t)\int_0^t du \, e^{\imath u(|\eta|-|\theta-\eta|)}\cf_x(\cg^n_{t-u})(\theta_1,\ldots,\theta_d)\int_u^{2u} ds \, e^{\imath  s\xi} \int_u^{2u} dr \, e^{-\imath r\xiti}\bigg|^2\nonumber\\
&\geq c \int_{(1,\infty)^d\cap \cac_d} \bigg(\prod_{i=1,\ldots,d}\frac{d\eta_i}{|\eta_i|^{4H'_i-2}}\bigg)\frac{1}{|\eta|^{4H'_0+2}}\1_{\{|\eta|\leq \frac{n}{2}\}}\nonumber\\
&\bigg[\int_{|\theta|\leq 1}d\theta\, \big|\cf_x(\psi)(\theta)\big|^2 \int_{0}^{1} d\xi\int_{0}^{1} d\xiti\, \bigg|\int_0^\infty dt \, \vp(t)\int_0^t du \, e^{\imath u(|\eta|-|\eta-\theta|)}\cf_x(\cg^n_{t-u})(\theta)\int_u^{2u} ds \, e^{\imath  s\xi} \int_u^{2u} dr \, e^{-\imath r\xiti}\bigg|^2\bigg]\, .\label{low-bound-i-m}
\end{align}
Let us recall the assumptions contained in Definition \ref{defi-space-ce}, namely $\vp\geq 0$, $\text{supp}\ \vp\subset (0,1)$, as well as
\begin{equation*}
c_\psi:=\inf_{|\theta|\leq 1}|\cf_x(\psi)(\theta)|>0\quad \text{and}\quad c_\vp:=\int_{\frac34}^{1} dt \, \vp(t) >0\, .
\end{equation*}
For every $\eta\in \R^d$, we get, using the above constants $c_\psi,c_\vp$,
\begin{align*}
&\int_{|\theta|\leq 1}d\theta\, \big|\cf_x(\psi)(\theta)\big|^2 \int_{0}^{1} d\xi\int_{0}^{1} d\xiti\, \bigg|\int_0^\infty dt \, \vp(t)\int_0^t du \, e^{\imath u(|\eta|-|\eta-\theta|)}\cf_x(\cg^n_{t-u})(\theta)\int_u^{2u} ds \, e^{\imath  s\xi} \int_u^{2u} dr \, e^{-\imath r\xiti}\bigg|^2\\
&\geq c_\psi^2\int_{|\theta|\leq 1}d\theta\, \int_{0}^{1} d\xi\int_{0}^{1} d\xiti\, \bigg|\int_0^1 dt \, \vp(t)\int_0^t du \, e^{\imath u(|\eta|-|\eta-\theta|)}\sin((t-u)|\theta|)\int_u^{2u} ds \, e^{\imath  s\xi} \int_u^{2u} dr \, e^{-\imath r\xiti}\bigg|^2\\
&\geq c_\psi^2\int_{\frac14\leq|\theta|\leq \frac12}d\theta\int_{0}^{\frac14} d\xi\int_{0}^{\frac14} d\xiti\, \bigg|\int_0^1 dt \, \vp(t)\int_0^t du\, \sin\big((t-u)|\theta|\big)\int_u^{2u} ds \int_u^{2u} dr \, \cos\big(u(|\eta|-|\eta-\theta|)+s\xi-r\xiti\big)\bigg|^2\\
&\geq c_\psi^2\frac{\cos^2(1)}{4^2}\int_{\frac14\leq |\theta|\leq \frac12}d\theta\bigg[\int_{\frac34}^1 dt \, \vp(t)\int_{\frac14}^{\frac12} du\, \sin\big((t-u)|\theta|\big)u^2 \bigg]^2\\
&\geq c_\psi^2 c_{\vp}^2 \frac{\cos^2(1)}{4^8}\sin^2\big(\frac{1}{16}\big)\int_{\frac14\leq |\theta|\leq \frac12}d\theta \ > \ 0\, .
\end{align*}

\

Going back to \eqref{low-bound-i-m}, we have thus shown the existence of a constant $c_0>0$ such that
\begin{align*}
\cj^n_M&\geq c_0 \int_{(1,\infty)^d\cap \cac_d} \bigg(\prod_{i=1,\ldots,d}\frac{d\eta_i}{|\eta_i|^{4H'_i-2}}\bigg)\frac{1}{|\eta|^{4H'_0+2}}\1_{\{|\eta|\leq \frac{n}{2}\}}.
\end{align*}
By recalling the definition \eqref{defi:ens-c-d} of $\cac_d$, we can use a spherical change of coordinates and obtain that for some constant $c_1>0$,
\begin{align*}
\cj^n_M&\geq c_1 \int_2^{\frac{n}{2}}\, \frac{dr}{r^{4(H'_0+H'_+)-3d+3}}\, .
\end{align*}
Following \eqref{h-prime}, one has precisely $4(H'_0+H'_+)-3d+3=1$, which entails the desired conclusion:
\begin{equation}\label{div-cj-m}
\lim_{n\to\infty}\cj^n_M = \infty \, .
\end{equation}

\subsection{Analysis of $\cj^n_R$}

Observe again that $\cf \rho_n(\xi,\eta)=\cf\rho(2^{-n} \xi,2^{-n}\eta)$, and so, just as in Section \ref{sec:item-i}, one can write $\|\cf \rho_n\|_{L^\infty}\leq \|\cf \rho\|_{L^\infty}\leq \|\rho\|_{L^1}=1$, which gives
\begin{align}
&\int_{|\eta|}^{2|\eta|} \frac{d\xi}{|\xi|^{2H'_0-1}}\int_{|\etati|}^{2|\etati|} \frac{d\xiti}{|\xiti|^{2H'_0-1}}\, \big|\cf \rho_n(-\xi,-\eta)\big|^2\big|\cf \rho_n(\xiti,\etati)\big|^2\big|R^n_\vp((\xi,\eta),(\xiti,\etati))\big|^2\nonumber\\
&\leq \int_{|\eta|}^{2|\eta|} \frac{d\xi}{|\xi|^{2H'_0-1}}\int_{|\etati|}^{2|\etati|} \frac{d\xiti}{|\xiti|^{2H'_0-1}}\, \big|R^n_\vp((\xi,\eta),(\xiti,\etati))\big|^2\, .\label{estim-i-n-r-1}
\end{align}
Using now the uniform estimate \eqref{bound-r-n}, we can assert that for all $\eta,\etati\in \R^d$ with $|\eta|\geq 1$ and $|\etati|\geq 1$, and for every $\ka\in [0,\frac12)$,
\begin{align}
&\int_{|\eta|}^{2|\eta|} \frac{d\xi}{|\xi|^{2H'_0-1}}\int_{|\etati|}^{2|\etati|} \frac{d\xiti}{|\xiti|^{2H'_0-1}}\big|R^n_\vp((\xi,\eta),(\xiti,\etati))\big|^2\nonumber\\
&=\frac{c}{|\eta|^{2H'_0-2} |\etati|^{2H'_0-2}} \int_{1}^{2} \frac{d\xi}{|\xi|^{2H'_0-1}}\int_{1}^{2} \frac{d\xiti}{|\xiti|^{2H'_0-1}}\big|R^n_\vp((\xi|\eta|,\eta),(\xiti|\etati|,\etati))\big|^2\nonumber\\
&\lesssim \frac{1}{|\eta|^{2H'_0} |\etati|^{2H'_0}} \int_{1}^{2} d\xi\int_{1}^{2} d\xiti\, \bigg|\frac{1}{|\eta|^\ka|\xi-1|^\ka |\etati||\xiti+1|}+\frac{1}{|\eta||\xi+1| |\etati||\xiti+1|}+\frac{1}{|\eta||\xi+1| |\etati|^\ka |\xiti-1|^\ka}\bigg|^2\nonumber\\
&\lesssim \frac{1}{|\eta|^{2H'_0} |\etati|^{2H'_0}} \int_{1}^{2} d\xi\int_{1}^{2} d\xiti\, \bigg[\frac{1}{|\eta|^{2\ka}|\xi-1|^{2\ka} |\etati|^2}+\frac{1}{|\eta|^2 |\etati|^{2}}+\frac{1}{|\eta|^2 |\etati|^{2\ka} |\xiti-1|^{2\ka}}\bigg]\nonumber\\
&\lesssim  \bigg[\frac{1}{|\eta|^{2H'_0+2\ka} |\etati|^{2H'_0+2}}+\frac{1}{|\eta|^{2H'_0+2} |\etati|^{2H'_0+2\ka} }\bigg]\, .\label{estim-i-n-r-2}
\end{align}
Injecting successively \eqref{estim-i-n-r-1} and \eqref{estim-i-n-r-2} into the expression \eqref{defi:i-n-r} of $\cj_R^n$, we obtain that for every $\ka\in [0,\frac12)$,
\begin{align}
\big|\cj^n_R\big|&\lesssim\int_{(1,\infty)^d\cap \cac_d} \bigg(\prod_{i=1,\ldots,d}\frac{d\eta_i}{|\eta_i|^{2H'_i-1}}\bigg)\int_{\eta_1}^{2\eta_1}\frac{d\etati_1}{|\etati_1|^{2H'_1-1}}\cdots\int_{\eta_d}^{2\eta_d}\frac{d\etati_d}{|\etati_d|^{2H'_d-1}} \big|\cf_x(\psi)(\eta-\etati)\big|^2\nonumber\\
&\hspace{7cm}\bigg[\frac{1}{|\eta|^{2H'_0+2\ka} |\etati|^{2H'_0+2}}+\frac{1}{|\eta|^{2H'_0+2} |\etati|^{2H'_0+2\ka} }\bigg],\label{transi-1}
\end{align}
and with Notation \eqref{nota:prod-vec}, this yields
\begin{align}
\big|\cj^n_R\big|
&\lesssim \int_{(1,\infty)^d\cap \cac_d} \bigg(\prod_{i=1,\ldots,d}\frac{d\eta_i}{|\eta_i|^{4H'_i-3}}\bigg)\frac{1}{|\eta|^{4H'_0+2+2\ka}}\int_{[1,2]^d}d\beta\,  \big|\cf_x(\psi)(\eta*_.(1-\be))\big|^2\nonumber\\
&\lesssim \int_{(1,\infty)^d\cap \cac_d} \bigg(\prod_{i=1,\ldots,d}\frac{d\eta_i}{|\eta_i|^{4H'_i-3}}\bigg)\frac{1}{|\eta|^{4H'_0+2+2\ka}}\int_{[0,1]^d}d\beta\,  \big|\cf_x(\psi)(\eta*_.\be)\big|^2\nonumber\\
&\lesssim \int_{(1,\infty)^d\cap \cac_d} \bigg(\prod_{i=1,\ldots,d}\frac{d\eta_i}{|\eta_i|^{4H'_i-2}}\bigg)\frac{1}{|\eta|^{4H'_0+2+2\ka}}\int_{0}^{\infty}d\theta_1\cdots\int_{0}^{\infty}d\theta_d\,  \big|\cf_x(\psi)(\theta)\big|^2\nonumber\\
&\lesssim \int_{(1,\infty)^d\cap \cac_d} \bigg(\prod_{i=1,\ldots,d}\frac{d\eta_i}{|\eta_i|^{4H'_i-2}}\bigg)\frac{1}{|\eta|^{4H'_0+2+2\ka}}.\label{transi-2}
\end{align}
Keeping in mind the definition \eqref{defi:ens-c-d} of $\cac_d$, we can then perform a spherical change of coordinates and get
\begin{align*}
\big|\cj^n_R\big|
&\lesssim \int_1^\infty dr \, \frac{r^{d-1}}{r^{4(H'_0+H'_+)-2d+2+2\ka}}\lesssim \int_1^\infty\, \frac{dr}{r^{1+2\ka}}\, ,
\end{align*}
where we have used the relation $4(H'_0+H'_+)-3d+3=1$. Picking $\ka\in (0,\frac12)$, this finally proves that
\begin{equation}\label{bound-unif-r-n}
\sup_{n\geq 1} \big|\cj^n_R\big| < \infty\, .
\end{equation}

\subsection{Analysis of $\cj^n_{M,R}$}

As above, we can use the fact that $\|\cf \rho_n\|_{L^\infty}\leq 1$ to obtain
\begin{align}
&\big|\cj^n_{M,R}\big|\nonumber\\
&\lesssim\int_{(1,\infty)^d\cap \cac_d} \bigg(\prod_{i=1,\ldots,d}\frac{d\eta_i}{|\eta_i|^{2H'_i-1}}\bigg)\int_{\eta_1}^{2\eta_1}\frac{d\etati_1}{|\etati_1|^{2H'_1-1}}\cdots\int_{\eta_d}^{2\eta_d}\frac{d\etati_d}{|\etati_d|^{2H'_d-1}} \big|\cf_x(\psi)(\eta-\etati)\big|^2\nonumber\\
&\hspace{5cm}\int_{|\eta|}^{2|\eta|} \frac{d\xi}{|\xi|^{2H'_0-1}}\int_{|\etati|}^{2|\etati|} \frac{d\xiti}{|\xiti|^{2H'_0-1}}\big|M^n_\vp((\xi,\eta),(\xiti,\etati))\big| \big|R^n_\vp((\xi,\eta),(\xiti,\etati))\big|\nonumber\\
&\lesssim\int_{(1,\infty)^d\cap \cac_d} \bigg(\prod_{i=1,\ldots,d}\frac{d\eta_i}{|\eta_i|^{2H'_i-1}}\bigg)\int_{\eta_1}^{2\eta_1}\frac{d\etati_1}{|\etati_1|^{2H'_1-1}}\cdots\int_{\eta_d}^{2\eta_d}\frac{d\etati_d}{|\etati_d|^{2H'_d-1}} \big|\cf_x(\psi)(\eta-\etati)\big|^2\nonumber\\
&\bigg(\int_{|\eta|}^{2|\eta|} \frac{d\xi}{|\xi|^{2H'_0-1}}\int_{|\etati|}^{2|\etati|} \frac{d\xiti}{|\xiti|^{2H'_0-1}}\big|M^n_\vp((\xi,\eta),(\xiti,\etati))\big|^2 \bigg)^{\frac12} \bigg(\int_{|\eta|}^{2|\eta|} \frac{d\xi}{|\xi|^{2H'_0-1}}\int_{|\etati|}^{2|\etati|} \frac{d\xiti}{|\xiti|^{2H'_0-1}} \big|R^n_\vp((\xi,\eta),(\xiti,\etati))\big|^2\bigg)^{\frac12},\label{estim-i-n-r-m-1}
\end{align}
where the second estimate is naturally derived from Cauchy-Schwarz inequality. As far as the integral of $M^n_\vp$ is concerned, let us write this time 
\begin{align*}
&\int_{|\eta|}^{2|\eta|} \frac{d\xi}{|\xi|^{2H'_0-1}}\int_{|\etati|}^{2|\etati|} \frac{d\xiti}{|\xiti|^{2H'_0-1}}\big|M^n_\vp((\xi,\eta),(\xiti,\etati))\big|^2\\
&=\frac{1}{|\eta|^{2H'_0-2} |\etati|^{2H'_0-2}} \int_{1}^{2} \frac{d\xi}{|\xi|^{2H'_0-1}}\int_{1}^{2} \frac{d\xiti}{|\xiti|^{2H'_0-1}}\big|M^n_\vp((\xi|\eta|,\eta),(\xiti|\etati|,\etati))\big|^2\\
&\lesssim \frac{1}{|\eta|^{2H'_0} |\etati|^{2H'_0}} \int_{1}^{2} d\xi\int_{1}^{2} d\xiti\, \bigg[\int_0^\infty dt \, \vp(t)\int_0^t du \, \big|\cf_x(\cg^n_{t-u})(\eta-\etati)\big|\bigg|\int_0^u ds \, e^{\imath  s|\eta|(\xi-1)} \int_0^u dr \, e^{-\imath r|\etati|(\xiti-1)}\bigg|\bigg]^2\\
&\lesssim \frac{1}{|\eta|^{2H'_0} |\etati|^{2H'_0}} \int_{0}^{1} d\xi\int_{0}^{1} d\xiti\, \bigg[\int_0^\infty dt \, \vp(t)\int_0^t du \, \bigg|\int_0^u ds \, e^{\imath  s|\eta|\xi} \int_0^u dr \, e^{-\imath r|\etati|\xiti}\bigg|\bigg]^2\\
&\lesssim\frac{1}{|\eta|^{2H'_0+1} |\etati|^{2H'_0+1}} \int_{0}^{|\eta|} d\xi\int_{0}^{|\etati|} d\xiti\, \bigg[\int_0^\infty dt \, \vp(t)\int_0^t du \, \bigg|\int_0^u ds \, e^{\imath  s\xi}\bigg|\bigg| \int_0^u dr \, e^{-\imath r\xiti}\bigg|\bigg]^2,
\end{align*}
and since $\text{supp} \, \vp\subset (0,1)$, this entails
\begin{align}
&\int_{|\eta|}^{2|\eta|} \frac{d\xi}{|\xi|^{2H'_0-1}}\int_{|\etati|}^{2|\etati|} \frac{d\xiti}{|\xiti|^{2H'_0-1}}\big|M^n_\vp((\xi,\eta),(\xiti,\etati))\big|^2\nonumber\\
&\lesssim\frac{\|\vp\|_\infty}{|\eta|^{2H'_0+1} |\etati|^{2H'_0+1}} \int_{0}^{\infty} d\xi\int_{0}^{\infty} d\xiti\, \bigg[\int_0^1 dt \, \int_0^t du \, \bigg|\int_0^u ds \, e^{\imath  s\xi}\bigg|\bigg| \int_0^u dr \, e^{-\imath r\xiti}\bigg|\bigg]^2\nonumber\\
&\lesssim\frac{\|\vp\|_\infty}{|\eta|^{2H'_0+1} |\etati|^{2H'_0+1}} \bigg(\int_{0}^\infty \frac{d\xi}{1+|\xi|^2}\bigg)^2\lesssim\frac{1}{|\eta|^{2H'_0+1}|\etati|^{2H'_0+1}}\, .\label{estim-i-n-r-m-2}
\end{align}
On the other hand, we have shown in \eqref{estim-i-n-r-2} that for every $\ka\in [0,\frac12)$,
\begin{align}
&\int_{|\eta|}^{2|\eta|} \frac{d\xi}{|\xi|^{2H'_0-1}}\int_{|\etati|}^{2|\etati|} \frac{d\xiti}{|\xiti|^{2H'_0-1}}\big|R^n_\vp((\xi,\eta),(\xiti,\etati))\big|^2\lesssim  \bigg[\frac{1}{|\eta|^{2H'_0+2\ka} |\etati|^{2H'_0+2}}+\frac{1}{|\eta|^{2H'_0+2} |\etati|^{2H'_0+2\ka} }\bigg]\, .\label{estim-i-n-r-m-3}
\end{align}
Injecting \eqref{estim-i-n-r-m-2} and \eqref{estim-i-n-r-m-3} into \eqref{estim-i-n-r-m-1}, we obtain that for every $\ka\in [0,\frac12)$, 
\begin{align*}
\big|\cj^n_{M,R}\big|&\lesssim\int_{(1,\infty)^d \cap \cac_d} \bigg(\prod_{i=1,\ldots,d}\frac{d\eta_i}{|\eta_i|^{2H'_i-1}}\bigg)\int_{\eta_1}^{2\eta_1}\frac{d\etati_1}{|\etati_1|^{2H'_1-1}}\cdots\int_{\eta_d}^{2\eta_d}\frac{d\etati_d}{|\etati_d|^{2H'_d-1}} \big|\cf_x(\psi)(\eta-\etati)\big|^2\\
&\hspace{3cm}\bigg(\frac{1}{|\eta|^{2H'_0+1} |\etati|^{2H'_0+1}}\bigg)^{\frac12} \bigg(\frac{1}{|\eta|^{2H'_0+2\ka} |\etati|^{2H'_0+2}}+\frac{1}{|\eta|^{2H'_0+2} |\etati|^{2H'_0+2\ka} }\bigg)^{\frac12}\\
&\lesssim\int_{(1,\infty)^d \cap \cac_d} \bigg(\prod_{i=1,\ldots,d}\frac{d\eta_i}{|\eta_i|^{2H'_i-1}}\bigg)\int_{\eta_1}^{2\eta_1}\frac{d\etati_1}{|\etati_1|^{2H'_1-1}}\cdots\int_{\eta_d}^{2\eta_d}\frac{d\etati_d}{|\etati_d|^{2H'_d-1}} \big|\cf_x(\psi)(\eta-\etati)\big|^2\\
&\hspace{6cm}\bigg[\frac{1}{|\eta|^{2H'_0+\frac12+\ka} |\etati|^{2H'_0+\frac32}}+\frac{1}{|\eta|^{2H'_0+\frac32} |\etati|^{2H'_0+\frac12+\ka} }\bigg]\\
&\lesssim \int_{(1,\infty)^d \cap \cac_d} \bigg(\prod_{i=1,\ldots,d}\frac{d\eta_i}{|\eta_i|^{4H'_i-2}}\bigg)\frac{1}{|\eta|^{4H'_0+2+\ka}},
\end{align*}
where, to derive the last estimate, we have used the same successive arguments as in the transition from \eqref{transi-1} to \eqref{transi-2}. Performing the same change of variables as in the previous situations, and also using the assumption $4(H'_0+H'_+)-3d+3=1$, we deduce that for every $\ka\in [0,\frac12)$,
\begin{align*}
\big|\cj^n_{M,R}\big|
&\lesssim \int_1^\infty dr \, \frac{r^{d-1}}{r^{4(H'_0+H'_+)-2d+2+\ka}}\lesssim \int_1^\infty\, \frac{dr}{r^{1+\ka}}\, .
\end{align*}
By picking $\ka>0$, this shows that
\begin{equation}\label{bound-unif-m-r-n}
\sup_{n\geq 1} \big|\cj^n_{M,R}\big| < \infty\, .
\end{equation}

\

\color{black}

\subsection{Conclusion}

By gathering \eqref{fir-low-bo}, \eqref{second-bo} and \eqref{decompo-cj-n}, we get first that
$$\mathbb{E}\Big[ \big|\langle \<IPsi2>^n,\Phi\rangle \big|^2\Big]\geq 2\big\{\cj_M^n+\cj_R^n+\cj_{M,R}^n\big\}\, .$$ 
Then, by \eqref{div-cj-m}, \eqref{bound-unif-r-n} and \eqref{bound-unif-m-r-n}, we know that $\lim_{n\to\infty}\cj^n_M = \infty$ and $\sup_{n\geq 1} \big|\cj_R^n+\cj_{M,R}^n\big| < \infty$, which yields the desired conclusion, namely: for every $\Phi\in \ce$ (where $\ce$ is the set introduced in Definition \ref{defi-space-ce}), it holds that
$$\lim_{n\to\infty} \mathbb{E}\Big[ \big|\langle \<IPsi2>^n,\Phi\rangle \big|^2\Big] =\infty\, .$$

\bigskip


\appendix

\section{Explosion of the fractional L{\'e}vy area}\label{append:levy-area}

We propose here to review the arguments behind Proposition \ref{prop:intro-fbm}, as a possible comparison with the estimates in the wave setting.

Let us recall first that a two-dimensional fractional noise $\dot{B}=(\dot{B}^{(1)},\dot{B}^{(2)})$ on $\R$, with index $H\in (0,1)$, is characterized by the covariance formula: for $1\leq i,j\leq 2$, and for all smooth compactly-supported functions $\psi,\vp:\R\to \R$,  
\begin{equation*}
\mathbb{E}\big[ \langle \dot{B}^{(i)},\phi\rangle \, \langle \dot{B}^{(j)},\vp\rangle \big]=\1_{\{i=j\}}\int_{\R} \frac{d\xi}{|\xi|^{2H-1}}\,  \cf \psi(\xi)  \overline{\cf \vp(\xi)} . 
\end{equation*}
Based on this expression, and following the proof of Lemma \ref{lem:cova}, it is then easy to check that the covariance formula for the mollified noise introduced in \eqref{mollifi-noi} reads as follows: for all $n\geq 1$, $1\leq i,j\leq 2$ and $s,t\in \R$,
\begin{equation}\label{cova-deriv-fbm}
\mathbb{E}\big[ \dot{B}^{(i),n}_s \dot{B}^{(j),n}_t \big]= \1_{\{i=j\}}\int_{\R} \frac{d\xi}{|\xi|^{2H-1}}\, |\cf \rho(-2^{-n}\xi)|^2 e^{\imath \xi (s-t)}\, .
\end{equation}
where $\cf \rho (\xi):=\int_{\R} dx \, e^{-\imath \xi x}\rho(x)$. We are now in a position to prove the statement.

\begin{proof}[Proof of Proposition \ref{prop:intro-fbm}]
As we mentionned it in the introduction, part $(i)$ of the proposition is actually a well-known result of rough paths theory (see for instance \cite[Theorem 15.33]{friz-victoir}).

\smallskip

Let us therefore concentrate on the proof of the divergence property, that is Part $(ii)$ of the proposition. To this end, we fix $H\in (0,\frac14]$, as well as a non-zero smooth function $\vp:\R_+\to \R_+$ with support in $[0,1]$. Based on the covariance formula \eqref{cova-deriv-fbm}, we can first expand the expectation under consideration as 
\begin{align*}
&\mathbb{E}\bigg[\Big|\big\langle t\mapsto \int_0^t ds\int_0^s dr\,  \dot{B}^{(1),n}_r \, \dot{B}^{(2),n}_s,\vp\big\rangle \Big|^2\bigg] \\
&=\mathbb{E}\bigg[\bigg|\int_0^\infty dt \, \vp(t)\int_0^t ds\int_0^s dr\,  \dot{B}^{(1),n}_r \, \dot{B}^{(2),n}_s \bigg|^2\bigg]\\
&=\int_{(0,\infty)^2} dtd\tti \, \vp(t)\vp(\tti)\int_0^t ds\int_0^s dr \int_0^{\tti} d\sti\int_0^{\sti}d\rti\, \mathbb{E}\Big[ \dot{B}^{(1),n}_r \dot{B}^{(2),n}_s \overline{\dot{B}^{(1),n}_{\rti}} \overline{\dot{B}^{(2),n}_{\sti}}\Big]\\
&=\int_{(0,\infty)^2} dtd\tti \, \vp(t)\vp(\tti)\int_0^t ds \int_0^{\tti} d\sti\, \mathbb{E}\Big[  \dot{B}^{(2),n}_s  \overline{\dot{B}^{(2),n}_{\sti}}\Big] \int_0^s dr \int_0^{\sti}d\rti\, \mathbb{E}\Big[ \dot{B}^{(1),n}_r \overline{\dot{B}^{(1),n}_{\rti}} \Big]\\
&=\int_{\R} \frac{d\xi}{|\xi|^{2H-1}}\, |\cf \rho(-2^{-n}\xi)|^2  \int_{\R} \frac{d\xiti}{|\xiti|^{2H-1}}\, |\cf \rho(-2^{-n}\xiti)|^2\\
&\hspace{4cm}  \int_{(0,\infty)^2} dtd\tti \, \vp(t)\vp(\tti)\int_0^t ds\int_0^{\tti} d\sti\, e^{\imath \xi (s-\sti)}\int_0^s dr \int_0^{\sti}d\rti\, e^{\imath \xiti (r-\rti)}\\
&= \int_{\R} \frac{d\xi}{|\xi|^{2H-1}}\, |\cf \rho(-2^{-n}\xi)|^2  \int_{\R} \frac{d\xiti}{|\xiti|^{2H+1}}\, |\cf \rho(-2^{-n}\xiti)|^2\\
&\hspace{4cm}  \int_{(0,\infty)^2} dtd\tti \, \vp(t)\vp(\tti)\int_0^t ds \int_0^{\tti} d\sti\, e^{\imath \xi (s-\sti)}\big\{e^{\imath \xiti s}-1\big\}\big\{e^{-\imath \xiti \sti}-1\big\}\\
&=\int_{\R} \frac{d\xi}{|\xi|^{2H-1}}\, |\cf \rho(-2^{-n}\xi)|^2  \int_{\R} \frac{d\xiti}{|\xiti|^{2H+1}}\, |\cf \rho(-2^{-n}\xiti)|^2\bigg|\int_0^\infty dt \, \vp(t)\int_0^t ds \, e^{\imath \xi s}\big\{e^{\imath \xiti s}-1\big\}\bigg|^2\\
&=\int_{\R} \frac{d\xi}{|\xi|^{2H-1}}\, |\cf \rho(-2^{-n}\xi)|^2  \int_{\R} \frac{d\xiti}{|\xiti|^{2H+1}}\, |\cf \rho(-2^{-n}\xiti)|^2 \big|M_\vp(\xi,\xiti)+R_\vp(\xi)\big|^2 \, ,
\end{align*}
where we have set
$$M_\vp(\xi,\xiti):=\int_0^\infty dt \, \vp(t)\int_0^t ds \, e^{\imath s(\xi+\xiti)} \quad , \quad R_\vp(\xi):=-\int_0^\infty dt \, \vp(t)\int_0^t ds \, e^{\imath \xi s}\, .$$
As a result, it holds that
\begin{equation}
\mathbb{E}\bigg[\Big|\big\langle t\mapsto \int_0^t ds\int_0^s dr\,  \dot{B}^{(1),n}_r \, \dot{B}^{(2),n}_s,\vp\big\rangle \Big|^2\bigg] \geq  \cj_n,
\end{equation}
where 
$$\cj_n:=\int_1^\infty \frac{d\xi}{|\xi|^{2H-1}}\, |\cf \rho(-2^{-n}\xi)|^2\int_{|\xi|}^{2|\xi|} \frac{d\xiti}{|\xiti|^{2H+1}}|\cf \rho(2^{-n}\xiti)|^2 \big|M_\vp(\xi,-\xiti)+R_\vp(\xi)\big|^2. $$
The integral $\cj_n$ can be further decomposed as $\cj_n=\cj_M^{n}+\cj_{M,R}^{n}$, with 
$$\cj_M^{n}:=\int_1^\infty \frac{d\xi}{|\xi|^{2H-1}}\, |\cf \rho(-2^{-n}\xi)|^2\int_{|\xi|}^{2|\xi|} \frac{d\xiti}{|\xiti|^{2H+1}}|\cf \rho(2^{-n}\xiti)|^2 \big|M_\vp(\xi,-\xiti)\big|^2$$
and
\begin{align}
&\cj_{M,R}^n:=\nonumber\\
&\int_1^\infty \frac{d\xi}{|\xi|^{2H-1}}\, |\cf \rho(-2^{-n}\xi)|^2\int_{|\xi|}^{2|\xi|} \frac{d\xiti}{|\xiti|^{2H+1}}|\cf \rho(2^{-n}\xiti)|^2 \big\{M_\vp(\xi,-\xiti)\overline{R_\vp(\xi)}+\overline{M_\vp(\xi,-\xiti)}R_\vp(\xi)+\big|R_\vp(\xi)\big|^2\big\}\, .\label{cj-m-r}
\end{align}

\

\noindent
\textit{Treatment of $\cj_{M,R}^n$.} Since $\|\cf \rho\|_\infty \leq \|\rho\|_{L^1(\R)}=1$, one has, uniformly over $n$,
\begin{align*}
\big|\cj_{M,R}^n\big|&\lesssim \int_1^\infty \frac{d\xi}{|\xi|^{2H-1}}\, \big|R_\vp(\xi)\big|\int_{|\xi|}^{2|\xi|} \frac{d\xiti}{|\xiti|^{2H+1}} \big|M_\vp(\xi,-\xiti)\big|+\int_1^\infty \frac{d\xi}{|\xi|^{2H-1}}\, \big|R_\vp(\xi)\big|^2\int_{|\xi|}^{2|\xi|} \frac{d\xiti}{|\xiti|^{2H+1}}\\
&\lesssim \int_1^\infty \frac{d\xi}{|\xi|^{4H-1}}\, \big|R_\vp(\xi)\big|\int_{1}^{2} \frac{d\beta}{|\beta|^{2H+1}} \big|M_\vp(\xi,-\beta |\xi|)\big|+\int_1^\infty \frac{d\xi}{|\xi|^{4H-1}}\, \big|R_\vp(\xi)\big|^2\int_{1}^{2} \frac{d\beta}{|\beta|^{2H+1}}.
\end{align*}
Now, it is clear that $\big|R_\vp(\xi)\big| \lesssim \frac{1}{|\xi|}$ and $\big|M_\vp(\xi,-\beta |\xi|)\big| \lesssim \frac{1}{|\xi|^{1-\varepsilon} |\be-1|^{1-\varepsilon}}$ for every $\varepsilon\in (0,1)$, which yields
\begin{align*}
\big|\cj_{M,R}^n\big|&\lesssim \int_1^\infty \frac{d\xi}{|\xi|^{4H+1-\varepsilon}}\, \int_{1}^{2} \frac{d\beta}{|\beta-1|^{1-\varepsilon}} +\int_1^\infty \frac{d\xi}{|\xi|^{4H+1}}.
\end{align*}
By picking $\varepsilon \in (0,4H)$, this shows that
$$\sup_{n\geq 0} \big|\cj_n^{M,R}\big| \ < \ \infty.$$

\

\noindent
\textit{Treatment of $\cj_{M}^n$.} Let us write, for every $\xi>1$,
\begin{align*}
\int_{|\xi|}^{2|\xi|} \frac{d\xiti}{|\xiti|^{2H+1}}\big|\cf \rho(2^{-n}\xiti)\big|^2 \big|M_\vp(\xi,-\xiti)\big|^2&= \frac{1}{|\xi|^{2H}}\int_1^2 d\beta \, \big|\cf \rho(2^{-n}\beta\xi)\big|^2\big|M_\vp(\xi,-\beta \xi)\big|^2\\
&= \frac{1}{|\xi|^{2H}}\int_1^2 d\beta \, \big|\cf \rho(2^{-n}\beta\xi)\big|^2\bigg|\int_0^\infty dt \, \vp(t)\int_0^t ds \, e^{-\imath s\xi(\beta-1)}\bigg|^2\\
&= \frac{1}{|\xi|^{2H}}\int_0^1 d\beta \, \big|\cf \rho(2^{-n}(\beta\xi+\xi))\big|^2\bigg|\int_0^\infty dt \, \vp(t)\int_0^t ds \, e^{-\imath s\xi\beta}\bigg|^2\\
&= \frac{1}{|\xi|^{2H+1}}\int_0^{\xi} d\theta \, \big|\cf \rho(2^{-n}(\theta+\xi))\big|^2 \bigg|\int_0^\infty dt \, \vp(t)\int_0^t ds \, e^{-\imath s\theta}\bigg|^2,
\end{align*}
and since we have assumed that $\text{supp}\, \vp\subset (0,1)$, we get
\begin{align*}
\int_{|\xi|}^{2|\xi|} \frac{d\xiti}{|\xiti|^{2H+1}}\big|\cf \rho(2^{-n}\xiti)\big|^2 \big|M_\vp(\xi,-\xiti)\big|^2
&\geq \frac{1}{|\xi|^{2H+1}}\int_0^{1} d\theta \, \big|\cf \rho(2^{-n}(\theta+\xi))\big|^2\bigg|\int_0^1 dt \, \vp(t)\int_0^t ds \, \cos( s\theta)\bigg|^2\\
&\geq \frac{1}{|\xi|^{2H+1}}\cos(1)^2\bigg[\int_0^1 dt \, \vp(t)t\bigg]^2\int_0^{1} d\theta \, \big|\cf \rho(2^{-n}(\theta+\xi))\big|^2\, .
\end{align*}
Injecting the latter bound into the expression of $\cj_n^{M}$, we obtain
\begin{align*}
\cj_n^{M} &\geq \cos(1)^2\bigg[\int_0^1 dt \, \vp(t)t\bigg]^2\int_1^\infty \frac{d\xi}{|\xi|^{4H}}\, |\cf \rho(-2^{-n}\xi)|^2\int_0^{1} d\theta \, \big|\cf \rho(2^{-n}(\theta+\xi))\big|^2\\
&\geq \cos(1)^2\bigg[\int_0^1 dt \, \vp(t)t\bigg]^2 \, 2^n\int_1^\infty \frac{d\xi}{|\xi|^{4H}}\, |\cf \rho(-2^{-n}\xi)|^2\int_{2^{-n}\xi}^{2^{-n}\xi+2^{-n}} d\theta \, \big|\cf \rho(\theta)\big|^2\\
&\geq \cos(1)^2\bigg[\int_0^1 dt \, \vp(t)t\bigg]^2 \, 2^{n(2-4H)}\int_{2^{-n}}^\infty \frac{d\xi}{|\xi|^{4H}}\, |\cf \rho(-\xi)|^2\int_{\xi}^{\xi+2^{-n}} d\theta \, \big|\cf \rho(\theta)\big|^2\, .
\end{align*}

Recall that $\cf \rho$ is assumed to be continuous and that $\cf \rho(0)=1$. Thus we can fix $\delta >0$ such that for every $|\xi|\leq \delta$, one has $\big|\cf \rho(\xi)\big|^2 \geq \frac12$. For every $n$ large enough so that $2^{-n}\leq \frac{\delta}{2}$, we get that
\begin{align*}
\cj_n^{M}&\geq \cos(1)^2\bigg[\int_0^1 dt \, \vp(t)t\bigg]^2\, 2^{n(2-4H)}\int_{2^{-n}}^{\frac{\delta}{2}} \frac{d\xi}{|\xi|^{4H}}\, |\cf \rho(-\xi)|^2\int_{\xi}^{\xi+2^{-n}} d\theta \, \big|\cf \rho(\theta)\big|^2\\
&\geq \frac{\cos(1)^2}{4}\bigg[\int_0^1 dt \, \vp(t)t\bigg]^2 \, 2^{n(1-4H)}\int_{2^{-n}}^{\frac{\delta}{2}} \frac{d\xi}{|\xi|^{4H}}\, .
\end{align*}
It is now clear that for every $H\leq \frac14$, the latter quantity tends to infinity as $n\to\infty$, and so 
\begin{equation*}
\cj_n^{M} \stackrel{n\to\infty}{\longrightarrow} \infty \, .
\end{equation*}

\

\noindent
\textit{Conclusion.} We have shown that 
$$\mathbb{E}\bigg[\Big|\big\langle t\mapsto \int_0^t ds\int_0^s dr\,  \dot{B}^{(1),n}_r \, \dot{B}^{(2),n}_s,\vp\big\rangle \Big|^2\bigg] \geq \cj_n^{M}+\cj_n^{M,R}, \quad \quad \text{with} \ \cj_n^{M} \stackrel{n\to\infty}{\longrightarrow} \infty \ \ \text{and} \ \ \sup_{n\geq 0} \big|\cj_n^{M,R}\big| < \infty,$$
which immediately entails the desired conclusion
$$\mathbb{E}\bigg[\Big|\big\langle t\mapsto \int_0^t ds\int_0^s dr\,  \dot{B}^{(1),n}_r \, \dot{B}^{(2),n}_s,\vp\big\rangle \Big|^2\bigg] \stackrel{n\to\infty}{\longrightarrow} \infty .$$
\end{proof}

\end{document}